\documentclass[12pt]{article}
\usepackage{latexsym, amssymb, amsthm}
\usepackage{dsfont}

\DeclareMathAlphabet\gothic{U}{euf}{m}{n}

\makeatletter
\def\eqnarray{\stepcounter{equation}\let\@currentlabel=\theequation
\global\@eqnswtrue
\tabskip\@centering\let\\=\@eqncr
$$\halign to \displaywidth\bgroup\hfil\global\@eqcnt\z@
  $\displaystyle\tabskip\z@{##}$&\global\@eqcnt\@ne
  \hfil$\displaystyle{{}##{}}$\hfil
  &\global\@eqcnt\tw@ $\displaystyle{##}$\hfil
  \tabskip\@centering&\llap{##}\tabskip\z@\cr}

\def\endeqnarray{\@@eqncr\egroup
      \global\advance\c@equation\m@ne$$\global\@ignoretrue}

\def\@yeqncr{\@ifnextchar [{\@xeqncr}{\@xeqncr[5pt]}}
\makeatother

\begin{document}
\bibliographystyle{tom}

\newtheorem{lemma}{Lemma}[section]
\newtheorem{thm}[lemma]{Theorem}
\newtheorem{cor}[lemma]{Corollary}
\newtheorem{prop}[lemma]{Proposition}
\newtheorem{ddefinition}[lemma]{Definition}
\newtheorem{stat}[lemma]{{\hspace{-5pt}}}

\theoremstyle{definition}

\newtheorem{remark}[lemma]{Remark}
\newtheorem{exam}[lemma]{Example}

\newcommand{\gota}{\gothic{a}}
\newcommand{\gotb}{\gothic{b}}
\newcommand{\gotc}{\gothic{c}}
\newcommand{\gote}{\gothic{e}}
\newcommand{\gotf}{\gothic{f}}
\newcommand{\gotg}{\gothic{g}}
\newcommand{\gothh}{\gothic{h}}
\newcommand{\gotk}{\gothic{k}}
\newcommand{\gotm}{\gothic{m}}
\newcommand{\gotn}{\gothic{n}}
\newcommand{\gotp}{\gothic{p}}
\newcommand{\gotq}{\gothic{q}}
\newcommand{\gotr}{\gothic{r}}
\newcommand{\gots}{\gothic{s}}
\newcommand{\gott}{\gothic{t}}
\newcommand{\gotu}{\gothic{u}}
\newcommand{\gotv}{\gothic{v}}
\newcommand{\gotw}{\gothic{w}}
\newcommand{\gotz}{\gothic{z}}
\newcommand{\gotA}{\gothic{A}}
\newcommand{\gotB}{\gothic{B}}
\newcommand{\gotG}{\gothic{G}}
\newcommand{\gotL}{\gothic{L}}
\newcommand{\gotS}{\gothic{S}}
\newcommand{\gotT}{\gothic{T}}

\newcounter{teller}
\renewcommand{\theteller}{(\alph{teller})}
\newenvironment{tabel}{\begin{list}%
{\rm  (\alph{teller})\hfill}{\usecounter{teller} \leftmargin=1.1cm
\labelwidth=1.1cm \labelsep=0cm \parsep=0cm}
                      }{\end{list}}

\newcounter{tellerr}
\renewcommand{\thetellerr}{(\roman{tellerr})}
\newenvironment{tabeleq}{\begin{list}%
{\rm  (\roman{tellerr})\hfill}{\usecounter{tellerr} \leftmargin=1.1cm
\labelwidth=1.1cm \labelsep=0cm \parsep=0cm}
                         }{\end{list}}

\newcounter{tellerrr}
\renewcommand{\thetellerrr}{(\Roman{tellerrr})}
\newenvironment{tabelR}{\begin{list}%
{\rm  (\Roman{tellerrr})\hfill}{\usecounter{tellerrr} \leftmargin=1.1cm
\labelwidth=1.1cm \labelsep=0cm \parsep=0cm}
                         }{\end{list}}

\newcounter{proofstep}
\newcommand{\nextstep}{\refstepcounter{proofstep}\vertspace \par 
          \noindent{\bf Step \theproofstep} \hspace{5pt}}
\newcommand{\firststep}{\setcounter{proofstep}{0}\nextstep}

\newcommand{\Ni}{\mathds{N}}
\newcommand{\Qi}{\mathds{Q}}
\newcommand{\Ri}{\mathds{R}}
\newcommand{\Ci}{\mathds{C}}
\newcommand{\Ti}{\mathds{T}}
\newcommand{\Zi}{\mathds{Z}}
\newcommand{\Fi}{\mathds{F}}

\renewcommand{\proofname}{{\bf Proof}}

\makeatletter
\renewenvironment{proof}[1][\proofname]{\par
  \pushQED{\qed}%
  \normalfont \topsep6\p@\@plus6\p@\relax
  \trivlist
  \item[\hskip\labelsep
        \itshape
    #1\@addpunct{{\bf.}}]\ignorespaces
}{%
  \popQED\endtrivlist\@endpefalse
}
\makeatother

\newcommand{\vertspace}{\vskip10.0pt plus 4.0pt minus 6.0pt}

\newcommand{\simh}{{\stackrel{{\rm cap}}{\sim}}}
\newcommand{\ad}{{\mathop{\rm ad}}}
\newcommand{\Ad}{{\mathop{\rm Ad}}}
\newcommand{\alg}{{\mathop{\rm alg}}}
\newcommand{\clalg}{{\mathop{\overline{\rm alg}}}}
\newcommand{\Aut}{\mathop{\rm Aut}}
\newcommand{\arccot}{\mathop{\rm arccot}}
\newcommand{\capp}{{\mathop{\rm cap}}}
\newcommand{\rcapp}{{\mathop{\rm rcap}}}
\newcommand{\diam}{\mathop{\rm diam}}
\newcommand{\divv}{\mathop{\rm div}}
\newcommand{\codim}{\mathop{\rm codim}}
\newcommand{\RRe}{\mathop{\rm Re}}
\newcommand{\IIm}{\mathop{\rm Im}}
\newcommand{\Tr}{{\mathop{\rm Tr}}}
\newcommand{\Vol}{{\mathop{\rm Vol}}}
\newcommand{\Leb}{{\mathop{\rm Leb}}}
\newcommand{\card}{{\mathop{\rm card}}}
\newcommand{\rank}{\mathop{\rm rank}}
\newcommand{\supp}{\mathop{\rm supp}}
\newcommand{\sgn}{\mathop{\rm sgn}}
\newcommand{\essinf}{\mathop{\rm ess\,inf}}
\newcommand{\esssup}{\mathop{\rm ess\,sup}}
\newcommand{\Int}{\mathop{\rm Int}}
\newcommand{\lcm}{\mathop{\rm lcm}}
\newcommand{\loc}{{\rm loc}}
\newcommand{\HS}{{\rm HS}}
\newcommand{\n}{{\rm N}}
\newcommand{\WOT}{{\rm WOT}}

\newcommand{\at}{@}

\newcommand{\mod}{\mathop{\rm mod}}
\newcommand{\spann}{\mathop{\rm span}}
\newcommand{\one}{\mathds{1}}

\hyphenation{groups}
\hyphenation{unitary}

\newcommand{\tfrac}[2]{{\textstyle \frac{#1}{#2}}}

\newcommand{\ca}{{\cal A}}
\newcommand{\cb}{{\cal B}}
\newcommand{\cc}{{\cal C}}
\newcommand{\cd}{{\cal D}}
\newcommand{\ce}{{\cal E}}
\newcommand{\cf}{{\cal F}}
\newcommand{\ch}{{\cal H}}
\newcommand{\chs}{{\cal HS}}
\newcommand{\ci}{{\cal I}}
\newcommand{\ck}{{\cal K}}
\newcommand{\cl}{{\cal L}}
\newcommand{\cm}{{\cal M}}
\newcommand{\cn}{{\cal N}}
\newcommand{\co}{{\cal O}}
\newcommand{\cp}{{\cal P}}
\newcommand{\cs}{{\cal S}}
\newcommand{\ct}{{\cal T}}
\newcommand{\cx}{{\cal X}}
\newcommand{\cy}{{\cal Y}}
\newcommand{\cz}{{\cal Z}}

\newlength{\hightcharacter}
\newlength{\widthcharacter}
\newcommand{\covsup}[1]{\settowidth{\widthcharacter}{$#1$}\addtolength{\widthcharacter}{-0.15em}\settoheight{\hightcharacter}{$#1$}\addtolength{\hightcharacter}{0.1ex}#1\raisebox{\hightcharacter}[0pt][0pt]{\makebox[0pt]{\hspace{-\widthcharacter}$\scriptstyle\circ$}}}
\newcommand{\cov}[1]{\settowidth{\widthcharacter}{$#1$}\addtolength{\widthcharacter}{-0.15em}\settoheight{\hightcharacter}{$#1$}\addtolength{\hightcharacter}{0.1ex}#1\raisebox{\hightcharacter}{\makebox[0pt]{\hspace{-\widthcharacter}$\scriptstyle\circ$}}}
\newcommand{\scov}[1]{\settowidth{\widthcharacter}{$#1$}\addtolength{\widthcharacter}{-0.15em}\settoheight{\hightcharacter}{$#1$}\addtolength{\hightcharacter}{0.1ex}#1\raisebox{0.7\hightcharacter}{\makebox[0pt]{\hspace{-\widthcharacter}$\scriptstyle\circ$}}}

\thispagestyle{empty}

\vspace*{1cm}
\begin{center}
{\Large\bf On one-parameter Koopman groups} \\[5mm]
\large A.F.M. ter Elst$^1$ and M. Lema\'nczyk$^2$

\end{center}

\vspace{5mm}

\begin{center}
{\bf Abstract}
\end{center}

\begin{list}{}{\leftmargin=1.8cm \rightmargin=1.8cm \listparindent=10mm 
   \parsep=0pt}
\item
We characterize Koopman one-parameter $C_0$-groups in the 
class of all unitary one-parameter $C_0$-groups on $L_2(X)$
as those that preserve $L_\infty(X)$ and for which 
the infinitesimal generator is a derivation on the bounded functions in its
domain.
\end{list}

\vspace{4cm}
\noindent
April 2015

\vspace{5mm}
\noindent
2010 AMS Subject Classification: 37A10, 47D03.

\vspace{5mm}
\noindent
Keywords: Koopman operator, 
one-parameter continuous family of measure-preserving transformations,
cocycle, non-singular transformation, consistent one-parameter groups.

\vspace{15mm}

\noindent
{\bf Home institutions:}    \\[3mm]
\begin{tabular}{@{}cl@{\hspace{10mm}}cl}
1. & Department of Mathematics  & 
  2. & Faculty of Mathematics and Computer Science \\
& University of Auckland   & 
  & Nicolaus Copernicus University  \\
& Private bag 92019 & 
  & 12/18 Chopin street \\
& Auckland 1142 & 
  &  87-100 Toru\'n \\
& New Zealand  & 
  & Poland  \\[10pt]
& terelst{\at}math.auckland.ac.nz
  & & mlem{\at}mat.umk.pl  
\end{tabular}

\newpage
\setcounter{page}{1}

\section{Introduction} \label{Skoop1}

Let $(X,\cb,\mu)$ be a standard Borel probability space.
Moreover, let $T \colon X\to X$ be an (a.e.) invertible, measurable
and measure-preserving map, i.e.\ $\mu(A)=\mu(T^{-1}A)$ for each $A\in\cb$.
Then $T$ induces
on $L_2(X)$ a unitary operator $U_T$, called a Koopman operator, defined by 
$U_Tf:=f\circ T$ for all $f\in L_2(X)$.
One can ask for the converse: given a unitary operator $U$ on $L_2(X)$, 
how to recognize that it is a Koopman operator.
The very classical answer says that if $U$  
preserves multiplication of bounded functions, i.e.\ if
\begin{equation}
U(f \, g) = U(f) \, U(g)
\label{eSkoop1;4}
\end{equation}
for all $f,g\in L_\infty(X)$, then $U$ is 
a Koopman operator by a combination of the multiplication theorem in \cite{Hal2} (page 45)
and \cite{Kec} Theorem~15.9.
Another type of questions one can ask for is, given a unitary operator 
$U$ on an abstract Hilbert space, how to recognize that it is unitarily equivalent 
to a Koopman operator, see for example \cite{CotlarRicabarra}, \cite{Choksi}, \cite{Rid}
and \cite{Denker}.

The problem which unitary operators can be realized as Koopman operator remains one of 
important and still unsolved problems in ergodic theory, see e.g.\ the 
discussion on this problem in \cite{KatokLeman}, \cite{KatokThouvenot} and also 
the survey article \cite{Leman1}.
Up to unitary equivalence each unitary operator $U$ is determined by the two invariants: 
the equivalence class $[\sigma]$ of a finite positive Borel measure $\sigma$ on the circle, 
called the maximal spectral type $\sigma_U$ of $U$, together with the
 (Borel) multiplicity functionÊ $M=M_U \colon \Ti \to \{1,2,\ldots\}\cup\{\infty\}$ 
which is defined $\sigma$-a.e.
Once a pair $([\sigma],M)$ is given, it is easy to construct on the abstract level 
a unitary operator $U$ for which $(\sigma_U,M_U)=([\sigma],M)$.
Nevertheless, it is an open problem whether there exists a (unitary) Koopman operator $U$
such that $(\sigma_U,M_U)=([\sigma],M)$.
(Some restrictions must be imposed on 
    $\sigma$, for example $\sigma$ must be of symmetric type and its topological 
    support must be full if the construction is sought in the class of 
    $U_T$ with $T$ ergodic.)
While some progress has been made recently in the spectral theory of 
single transformation, cf.\ \cite{Leman1}, for unitary one-parameter 
groups still little is known.

A unitary one-parameter $C_0$-group $(U_t)_{t\in\Ri}$ is called a Koopman group
if for all $t\in\Ri$ there exists a measurable 
$T_t \colon X\to X$ such that $U_tf=f\circ T_t$ for all $f\in L_2(X)$.
It is clear that a Koopman group must preserve $L_\infty(X)$, but this 
latter condition is satisfied also for many unitary one-parameter 
$C_0$-groups which are not Koopman groups.
By the Stone theorem \cite{Sto3}, the generator 
$A$ of a unitary one-parameter $C_0$-group is skew-adjoint.
Therefore each unitary one-parameter $C_0$-groupÊ is determined up to unitary equivalence
by $(\sigma_U,M_U)$, where $\sigma_U=[\sigma]$ for some finite positive Borel measure 
on $\Ri$.
In order to characterize those pairs $([\sigma],M)$ which can 
be realized by Koopman groups, it seems to natural to 
characterize first those generators $A$ for which 
$(e^{tA})_{t\in\Ri}$ is equivalent to a Koopman group.
Even the problem to characterize in terms of their generator
which unitary one-parameter $C_0$-groups 
are Koopman groups  seems to be, however, far from obvious.
Moreover, once such a characterization is done, one can consider the 
problem whether a perturbation of a Koopman representation 
remains Koopman.
The latter is of independent interest.

In order to formulate the main results of the paper, first recall that
if $A$ is an operator in a function space $E$ and 
$\cd \subset D(A)$ is an algebra,
then we say that $A$ is a {\bf derivation on $\cd$} if
\[
A(f \, g) = (A f) \, g + f \, (Ag)
\]
for all $f,g \in \cd$.
The main result of the paper is the following.

\begin{thm} \label{tkoop101}
Let $(X,\cb,\mu)$ be a standard Borel probability space.
Let $U$ be a unitary one-parameter $C_0$-group on $L_2(X)$
with generator $A$.
Then the following are equivalent.
\begin{tabeleq}
\item \label{tkoop101-1}
For all $t \in \Ri$ there exists an a.e.\ invertible measurable and measure preserving
map $T_t \colon X \to X$ such that $U_t f = f \circ T_t$
for all $f \in L_2(X)$.
\item \label{tkoop101-2}
The space $L_\infty(X)$ is invariant under $U$.
Moreover, the space $D(A) \cap L_\infty(X)$ is an algebra
and $A$ is a derivation on $D(A) \cap L_\infty(X)$.
\end{tabeleq}
\end{thm}

We are also able to prove in Corollary~\ref{ckoop207.5}
a generalisation of the above theorem where 
the group $U$ is a $C_0$-group which is not necessarily unitary
and we do not require measure preserving in Condition~\ref{tkoop101-1}.
Moreover we have a generalisation where 
the measure $\mu$ is merely $\sigma$-finite, see Theorem~\ref{tkoop207} below.

A theorem of the same nature as Theorem~\ref{tkoop101}
was given by Gallavotti and Pulvirenti, \cite{GalP} Theorem~4.

\begin{thm}[\cite{GalP}]
Let $(X,\cb,\mu)$ be a standard Borel probability space.
Let $A$ be a self-adjoint operator and let $\cd \subset D(A) \subset L_\infty(X)$.
Suppose that $\cd$ is a core for $A$, $\one \in \cd$, 
$\cd$ is an algebra, $\cd$ is self-adjoint 
(that is if $f \in \cd$ then $\overline f \in \cd$),
$A$ is a derivation on $\cd$ and $\overline{A f} = - A \overline f$
for all $f \in \cd$.
Then {\rm \ref{tkoop101-1}} in Theorem~{\rm \ref{tkoop101}} is valid.
\end{thm}

The theorem of Gallavotti and Pulvirenti does not have an extension 
where $A$ is merely a $C_0$-group generator and it is essential in \cite{GalP}
that the measure $\mu$ is finite.

The main application of Theorems~\ref{tkoop101} and \ref{tkoop207}
is a characterization of those $C_0$-groups on $L_2(X)$ which 
are pointwise the product of a Koopman operator and a multiplication operator.

\begin{thm} \label{tkoop311}
Let $(X,\cb,\mu)$ be a standard Borel probability space.
Let $U$ be a unitary $C_0$-group on $L_2(X)$ preserving $L_\infty(X)$.
Assume that $\one\in D(A)$ with $A\one\in L_\infty(X)$, where $A$ is the generator of $U$.
Then the following are equivalent.
\begin{tabeleq}
\item \label{tkoop311-1}
For all $t\in\Ri$ 
there exists an a.e.\ invertible, measurable
and measure-preserving map $T_t \colon X\to X$ and a function 
$\psi_t \colon X\to \Ci$ such that 
$U_tf=\psi_t\cdot (f\circ T_t)$ for all $f\in L_2(X)$.
\item \label{tkoop311-2}
For all $t \in \Ri$ one has $|U_t \one| = 1$ a.e.
Moreover, 
$D(A)\cap L_\infty(X)$  is an algebra and
$A - (A\one) I$ is a derivation on $D(A)\cap L_\infty(X)$.
\end{tabeleq}
\end{thm}

We also have an extension of this theorem for $C_0$-groups which are not necessarily 
unitary, see Theorem~\ref{tkoop310}.
The above result can be viewed as the one-parameter counterpart of the classical 
Banach--Lamperti theorem, \cite{Lamperti} Theorem~3.1, classifying that all isometries of 
$L_p(X)$ for all $p \in [1,\infty) \setminus \{ 2 \} $ are of the form 
\begin{equation}
f\mapsto \psi\cdot (f\circ T)
\label{eSkoop1;6}
\end{equation} 
for some pointwise map $T \colon X\to X$ and 
$\psi \colon X\to (0,\infty)$.
In \cite{GGM} the authors also proved that unitary positivity preserving maps are 
of the form (\ref{eSkoop1;6}).

In Section~\ref{Skoop2} we prove Theorem~\ref{tkoop101} and its 
extension for general $C_0$-groups.
As a tool and byproduct 
we prove in Theorem~\ref{tkoop206} that if $(X,\cb,\mu)$ is a finite 
measure space and $S$ is a $C_0$-group in $L_2(X)$, then 
$S$ extends consistently to a $C_0$-group on $L_1(X)$ if and only if
the dual group $S^*$ leaves $L_\infty(X)$ invariant.
This is a new result in (semi)group theory.
In Section~\ref{Skoop3} we prove Theorem~\ref{tkoop310},
characterizing weighted non-singular one-parameter $C_0$-groups,
which has Theorem~\ref{tkoop311} as corollary.
It turns out that in Theorem~\ref{tkoop311}\ref{tkoop311-1} one has 
$\psi_t \in L_\infty(X)$ and 
\[
\psi_{t+s} = \psi_t \cdot (\psi_s \circ T_t) \mbox{ a.e.}
\]
for all $t,s \in \Ri$.
Finally, in Section~\ref{Skoop4} we determine the form of such 
$\psi$, assuming a differentiability condition.

\section{Derivations} \label{Skoop2}

If $(X,\cb,\mu)$ is a measure space and $f,g \in L_2(X)$, then
we denote the inner product by 
$(f,g) = \int_X f \, \overline g \, d\mu$.
Moreover, if $f \in L_\infty(X)$ and $g \in L_1(X)$ then we
denote the duality by
$\langle f, g \rangle = \int_X f \, \overline g \, d\mu$.
If the measure space is clear from the context, then we abbreviate
$L_p = L_p(X)$ for all $p \in [1,\infty]$.
Further, let $p,q \in [1,\infty]$, let $U$ be a one-parameter (semi)group on $L_p(X)$
and $V$ be a one-parameter (semi)group on $L_q(X)$.
We say that $U$ and $V$ are {\bf consistent} if
$U_t f = V_t f$ for all $t \in \Ri$ (or $t \in (0,\infty)$)
and $f \in L_p(X) \cap L_q(X)$.

For the proof of Theorem~\ref{tkoop101} we need several lemmas
as preparation.
The first two seem to be folklore.

\begin{lemma} \label{lkoop203}
Let $(X,\cb,\mu)$ be a measure space, $c > 0$ and $E \colon L_2(X) \to L_2(X)$
be a bounded operator.
Suppose that $\|E f\|_\infty \leq c \, \|f\|_\infty$ for all
$f \in L_2(X) \cap L_\infty(X)$.
Then there exist unique $\widehat E \in \cl(L_1(X))$ and
$\widetilde E \in \cl(L_\infty(X))$ such that
$\widehat Ef = E^* f$ for all $f \in L_1(X) \cap L_2(X)$
and $\widetilde Ef = E f$ for all $f \in L_\infty(X) \cap L_2(X)$.
Moreover, $\|\widetilde E\|_{\infty \to \infty} \leq c$ and
$\widetilde E = (\widehat E)^*$.
\end{lemma}
\begin{proof}
Let $f \in L_1 \cap L_2$.
Then $|(E^* f, g)| = |(f, Eg)| \leq \|f\|_1 \, \|E g\|_\infty \leq c \, \|f\|_1 \, \|g\|_\infty$
for all $g \in L_2 \cap L_\infty$.
So $E^* f \in L_1$ and $\|E^* f\|_1 \leq c \, \|f\|_1$.
Since $L_1 \cap L_2$ is dense in $L_1$ it follows that there
exists a unique $\widehat E \in \cl(L_1)$ such that
$\widehat Ef = E^* f$ for all $f \in L_1 \cap L_2$.
Choose $\widetilde E = (\widehat E)^*$.
Then the existence follows.
The uniqueness on $L_\infty$ is a consequence of the
w$^*$-density of $L_2 \cap L_\infty$ in $L_\infty$.
\end{proof}

As a consequence one has the next lemma.

\begin{lemma} \label{lkoop204}
Let $(X,\cb,\mu)$ be a measure space
and $S$ a semigroup on $L_2(X)$.
Suppose that for all $t \in (0,1]$ there exists a $c > 0$ 
such that $\|S_t f\|_\infty \leq c \, \|f\|_\infty$
for all $f \in L_2(X) \cap L_\infty(X)$.
Then there exist a unique semigroup $\widetilde S$ on $L_\infty(X)$
and a unique semigroup $\widehat S$ on $L_1(X)$
such that $\widetilde S$ is consistent with $S$
and $\widehat S$ is consistent with $S^*$.
Moreover, if there exists a $\tilde c \geq 1$ such that 
$\|S_t f\|_\infty \leq \tilde c \, \|f\|_\infty$
for all $t \in (0,1]$ and $f \in L_2(X) \cap L_\infty(X)$, then
$\|\widetilde S_t\|_{\infty \to \infty} = \|\widehat S_t\|_{1 \to 1} \leq c \, e^{t \log c}$
for all $t \in (0,\infty)$.
\end{lemma}

The next lemma is less known.

\begin{lemma} \label{lkoop205}
Let $(X,\cb,\mu)$ be a measure space, $c \geq 1$
and $S$ a $C_0$-semigroup on $L_2(X)$ with generator $A$.
Suppose that $\|S_t f\|_\infty \leq c \, \|f\|_\infty$
for all $t \in (0,1]$ and $f \in L_2(X) \cap L_\infty(X)$.
Then $D(A) \cap L_\infty(X)$ is dense in $L_2(X)$.
\end{lemma}
\begin{proof}
Since $L_2 \cap L_\infty$ is dense in $L_2$, it suffices to show that
for all $f \in L_2 \cap L_\infty$ there exists a sequence $(f_n)_{n \in \Ni}$
in $D(A) \cap L_\infty$ such that $\lim f_n = f$ in $L_2$.
Fix $\varphi \in C_c^\infty(0,\infty)$ with $\int \varphi = 1$.
For all $n \in \Ni$ define $\varphi \in C_c^\infty(0,\infty)$ by
$\varphi_n(t) = n \, \varphi(n \, t)$.
Let $f \in L_2 \cap L_\infty$ and $n \in \Ni$.
Define $f_n \in L_2$ by
\[
f_n = \int_{(0,\infty)} \varphi_n(t) \, S_t f \, dt
.  \]
Then $f_n \in D(A)$.
Moreover, $\lim_{n \to \infty} f_n = f$ in $L_2$ since $S$ is a continuous
semigroup.
It remains to show that $f_n \in L_\infty$ for all $n \in \Ni$.
Let $n \in \Ni$ and $g \in L_1 \cap L_2$.
Then
\[
|\varphi_n(t) \, (S_t f,g)|
\leq |\varphi_n(t)| \, \|S_t f\|_\infty \, \|g\|_1
\leq |\varphi_n(t)| \, c \, e^{t \log c} \, \|f\|_\infty \, \|g\|_1
\]
for all $t \in (0,\infty)$.
Hence
\[
|(f_n,g)|
= \Big| \int_{(0,\infty)} \varphi_n(t) \, (S_t f,g) \, dt \Big|
\leq M_n \, \|f\|_\infty \, \|g\|_1
,  \]
where $M_n = \int_{(0,\infty)} |\varphi_n(t)| \, c \, e^{t \log c} \, dt < \infty$.
So $f_n \in L_\infty$ as required.
\end{proof}

\begin{remark} \label{rkoop206}
Under the conditions of Lemma~\ref{lkoop205} the space $D(A) \cap L_\infty$ is
even a core for $A$, since the space $D(A) \cap L_\infty$ is invariant
under $S$.
See \cite{EN} Proposition II.1.7.
\end{remark}

It seems that the next theorem is new.
Note that we do not assume a uniform bound of the type (\ref{etkoop206;10})
in Condition~\ref{tkoop206-2}.

\begin{thm} \label{tkoop206}
Let $(X,\cb,\mu)$ be a finite measure space.
Let $S$ be a $C_0$-group on $L_2(X)$.
Then the following are equivalent.
\begin{tabeleq}
\item \label{tkoop206-1}
The group $S$ extends consistently to a $C_0$-group on $L_1(X)$.
\item \label{tkoop206-2}
The space $L_\infty(X)$ is invariant under $S^*$, that is
$S_t^* (L_\infty(X)) \subset L_\infty(X)$ for all $t \in \Ri$.
\end{tabeleq}
If these conditions are satisfied, then there exist $M\geq 1$ and $\omega \geq 0$
such that 
\begin{equation}
\|S_t^* f\|_\infty \leq M \, e^{\omega |t|} \, \|f\|_\infty
\label{etkoop206;10}
\end{equation}
for all $t \in \Ri$ and $f \in L_\infty(X)$.
\end{thm}
\begin{proof}
The implication `\ref{tkoop206-1}$\Rightarrow$\ref{tkoop206-2}' is trivial.
(Cf.\ the proof of Lemma~\ref{lkoop203}.)
So it remains to show `\ref{tkoop206-2}$\Rightarrow$\ref{tkoop206-1}'.

Let $t \in \Ri$.
It follows from the closed graph theorem that there exists a $c > 0$
such that $\|S_t^* f\|_\infty \leq c \, \|f\|_\infty$
for all $f \in L_\infty$.
Note that we use here that the measure $\mu$ is finite.
Also note that $c$ depends on $t$.
Hence by Lemma~\ref{lkoop204} there exists a one-parameter group
$\widehat S$ on $L_1$ and a one-parameter group
$\widetilde S$ on $L_\infty$ such that $\widehat S$ is consistent with $S$
and $\widetilde S$ is consistent with $S^*$.
Moreover, $\widetilde S_t = (\widehat S_t)^*$ for all $t \in \Ri$.

We shall show that $ \{ \widetilde S_t : t \in [2,3] \} $ is
bounded in $\cl(L_\infty)$.
By the uniform boundedness principle if suffices to show that
$ \{ \|\widetilde S_t f\|_\infty : t \in [2,3] \} $ is bounded for all $f \in L_\infty$.
For this we use the arguments as in the first step of the proof
of \cite{ABHN} Lemma~3.16.4.
Fix $f \in L_\infty$.
If $t \in \Ri$ then
\begin{eqnarray*}
\|\widetilde S_t f\|_\infty
& = & \sup \{ |\langle \widetilde S_t f, g\rangle| : g \in L_1 \mbox{ and } \|g\|_1 \leq 1 \}  \\
& = & \sup \{ |\langle \widetilde S_t f, g\rangle| : g \in L_2 \mbox{ and } \|g\|_1 \leq 1 \}  \\
& = & \sup \{ |(f, S_t g)| : g \in L_2 \mbox{ and } \|g\|_1 \leq 1 \}
.
\end{eqnarray*}
For each $g \in L_2$ the map
$t \mapsto |(f, S_t g)|$ is continuous by the strong continuity of $S$ on $L_2$.
So $t \mapsto \|\widetilde S_t f\|_\infty$ is lower semicontinuous and therefore  a measurable
function on $\Ri$.
This is the key assumption in the first step of the proof
of \cite{ABHN} Lemma~3.16.4.
In order to have the paper self-contained, we include the proof, with
minor modifications.
Suppose that $ \{ \|\widetilde S_t f\|_\infty : t \in [2,3] \} $ is not bounded.
Then there are $t_0,t_1,t_2,\ldots \in [2,3]$ such that $\lim_{n \to \infty} t_n = t_0$
and $\|\widetilde S_{t_n} f\|_\infty \geq n$ for all $n \in \Ni$.
Since $t \mapsto \|\widetilde S_t f\|_\infty$ is measurable, there are
$M > 0$ and a measurable set $F \subset [0,t_0]$ such that $\mu(F) > 1$ and
$\|\widetilde S_t f\|_\infty \leq M$ for all $t \in F$.
Let $n \in \Ni$.
Then
\[
n
\leq \|\widetilde S_{t_n} f\|_\infty
\leq \|\widetilde S_{t_n - t}\| \, \|\widetilde S_t f\|_\infty
\leq M \, \|\widetilde S_{t_n - t}\|
\]
for all $t \in F$.
So $\|\widetilde S_s\| \geq M^{-1} \, n$ for all $s \in E_n$, where
\[
E_n = \{ t_n - t : t \in F \cap [0,t_n] \}
.  \]
Note that $E_n$ is measurable and $\mu(E_n) \geq 1$ if
$|t_n - t_0| < \mu(F) - 1$.
Let $E = \limsup_{n \to \infty} E_n = \bigcap_{m=1}^\infty \bigcup_{n=m}^\infty E_n$.
Then $E$ is measurable and $\mu(E) \geq 1$.
In particular, $E \neq \emptyset$.
Moreover, $\|\widetilde S_s\| = \infty$ for all $s \in E$.
This is a contradiction.

Thus $ \{ \widetilde S_t : t \in [2,3] \} $ is
bounded in $\cl(L_\infty)$.
Since $\widetilde S_t = (\widehat S_t)^*$ it follows that
$ \{ \widehat S_t : t \in [2,3] \} $ is bounded in $\cl(L_1)$.
By the group property the set
$ \{ \widehat S_t : t \in [-1,1] \} $ is also bounded in $\cl(L_1)$.
Let $c = \sup \{ \|\widehat S_t\| : t \in [-1,1] \} < \infty$.
Let $g \in L_\infty$.
Then
$\lim_{t \to 0} \langle g, \widehat S_t f\rangle
= \lim_{t \to 0} (g, S_t f) = (g,f) = \langle g, f\rangle$
for all $f \in L_2$.
Since $L_2$ is dense in $L_1$ and $c < \infty$ it follows that
$\lim_{t \to 0} \langle g, \widehat S_t f\rangle = \langle g, f\rangle$
for all $f \in L_1$.
So $\widehat S$ is weakly continuous and hence $\widehat S$ is a $C_0$-group.
Finally, there are $M\geq 1$ and $\omega \geq 0$
such that $\|\widehat S_t\|_{1 \to 1} \leq M \, e^{\omega |t|}$ for all $t \in \Ri$.
Then
$\|\widetilde S_t\|_{\infty \to \infty} 
= |\widehat S_t\|_{1 \to 1}
\leq M \, e^{\omega |t|}$ for all $t \in \Ri$.
\end{proof}

We now turn to the proof of Theorem~\ref{tkoop101}.
The implication \ref{tkoop101-1}$\Rightarrow$\ref{tkoop101-2} in
Theorem~\ref{tkoop101} is a special case of the next proposition.

\begin{prop} \label{pkoop201}
Let $U$ be a $C_0$-group on $L_2(X)$
with generator $A$, where $(X,\cb,\mu)$ is a measure space.
Suppose that for every $t \in \Ri$ there exists a measurable
map $T_t \colon X \to X$ such that $U_t f = f \circ T_t$
for all $f \in L_2(X)$.
Then $U$ extends consistently to a w$^*$-continuous contraction
group on $L_\infty(X)$.
Moreover, $D(A) \cap L_\infty(X)$ is an algebra
and $A$ is a derivation on $D(A) \cap L_\infty(X)$.
\end{prop}
\begin{proof}
Let $t \in \Ri$.
If $f \in L_2 \cap L_\infty$ then
$\|U_t f\|_\infty = \|f \circ T_t\|_\infty \leq \|f\|_\infty$.
Hence by Lemma~\ref{lkoop204} there exist
a unique group $\widetilde U$ on $L_\infty$
and a unique group $\widehat U$ on $L_1$
such that $\widetilde U$ is consistent with $U$
and $\widehat U$ is consistent with $U^*$.
Moreover,
$\|\widetilde U_t\|_{\infty \to \infty} = \|\widehat U_t\|_{1 \to 1} = 1$
for all $t \in \Ri$.
Then $\widehat U$ is a $C_0$-group on $L_1$ by \cite{Voi} Proposition~4.
Therefore $\widetilde U$ is a w$^*$-continuous group on $L_\infty$.

Let $f,g \in D(A) \cap L_\infty$.
Then $U_t(f \, g) = (U_t f)(U_t g)$ for all $t \in \Ri$.
Hence
\[
\tfrac{1}{t} \Big( U_t(f \, g) - f \, g \Big)
= \tfrac{1}{t} \Big( U_t f - f \Big) U_t g
   + \tfrac{1}{t} \, f \, \Big( U_t g - g \Big)
\]
for all $t > 0$.
So $f \, g \in D(A)$ and $A(f \, g) = (A f) \, g + f \, (Ag)$.
\end{proof}

Under more conditions there is a converse of Proposition~\ref{pkoop201}.
A first step is the next proposition.

\begin{prop} \label{pkoop202}
Let $U$ be a $C_0$-group on $L_2(X)$
with generator $A$, where $(X,\cb,\mu)$ is a measure space.
Suppose that $D(A) \cap L_\infty(X)$ is an algebra
and $A$ is a derivation on $D(A) \cap L_\infty(X)$.
Moreover, suppose that there exists a $c \geq 1$ such that
\[
\|U_t f\|_\infty
\leq c \, \|f\|_\infty
\]
for all $t \in [-1,1]$ and $f \in L_2(X) \cap L_\infty(X)$.
Then there exists a unique one-parameter group $\widetilde U$ on $L_\infty(X)$
which is consistent with $U$.
Moreover,
\[
\widetilde U_t (f \, g) = (\widetilde U_t f) (\widetilde U_t g)
\]
for all $f,g \in L_\infty$ and $t \in \Ri$.
\end{prop}
\begin{proof}
The existence and uniqueness of the one-parameter group $\widetilde U$ on $L_\infty$
follows from Lemma~\ref{lkoop204}.
Then
$\|\widetilde U_t f\|_\infty \leq c \, e^{\omega |t|} \, \|f\|_\infty$
for all $t \in \Ri$ and $f \in L_\infty$, where $\omega = \log c$.

Clearly $t \mapsto \langle \widetilde U_t f,g \rangle = (U_t f, g)$
is continuous for all $f \in L_\infty \cap L_2$ and $g \in L_1 \cap L_2$.
Let $f \in L_\infty \cap L_2$ and $g \in L_1$.
Let $t \in \Ri$ and $\varepsilon > 0$.
There exists a $g' \in L_1 \cap L_2$ such that $\|g-g'\|_1 \leq \varepsilon$.
Then
\begin{eqnarray*}
|\langle \widetilde U_{t+k} f, g \rangle - \langle \widetilde U_t f, g \rangle|
& = & | \langle \widetilde U_{t+k} f, g-g' \rangle
   + \langle \widetilde U_{t+k} f - \widetilde U_t f, g' \rangle
   + \langle \widetilde U_t f, g'-g \rangle |  \\
& \leq & c \, e^{\omega(|t| + 1)} \, \varepsilon \, \|f\|_\infty
   + |\langle \widetilde U_{t+k} f - \widetilde U_t f, g' \rangle|
   + c \, e^{\omega |t|} \, \varepsilon \, \|f\|_\infty
\end{eqnarray*}
for all $k \in [-1,1]$.
Hence
\begin{equation}
\lim_{k \to 0} \langle \widetilde U_{t+k} f, g \rangle
= \langle \widetilde U_t f, g \rangle
\label{epkoop202;6}
\end{equation}
for all $f \in L_\infty \cap L_2$, $g \in L_1$ and $t \in \Ri$.

Let $f,g \in D(A) \cap L_\infty$.
Define $\alpha \colon \Ri \to L_2$ by
$\alpha(t) = (U_t f) (U_t g)$.
Let $h \in L_2$.
We shall show that $t \mapsto (\alpha(t),h)$ is differentiable and that
\begin{equation}
\frac{d}{dt} (\alpha(t), h)
= \Big( (A U_t f) (U_t g) + (U_t f) (A U_t g), h\Big)
\label{epkoop202;7}
\end{equation}
for all $t \in \Ri$.
Let $t \in \Ri$.
If $k \in \Ri \setminus \{ 0 \} $, then
\begin{equation}
\tfrac{1}{k} \Big( (\alpha(t+k), h) - (\alpha(t), h) \Big)
= \tfrac{1}{k} \Big( (U_{t+k} f - U_t f) U_{t+k} g, h \Big)
   + \tfrac{1}{k} \Big( (U_t f) ( U_{t+k} g - U_t g ), h \Big)
.
\label{epkoop202;5}
\end{equation}
For the first term we shall prove that
\[
\lim_{k \to 0} \tfrac{1}{k} \Big( (U_{t+k} f - U_t f) U_{t+k} g, h \Big)
= \Big( (A U_t f) (U_t g) , h\Big)
.  \]
Let $k \in \Ri \setminus \{ 0 \} $.
Then
\begin{eqnarray*}
\lefteqn{
\Big| \tfrac{1}{k} \Big( (U_{t+k} f - U_t f) U_{t+k} g, h \Big)
   - \Big( (A U_t f) (U_t g) , h\Big) \Big|
} \hspace{20mm} \\*
& = & \Bigg| \Bigg( \Big( \tfrac{1}{k} (U_{t+k} f - U_t f) - A U_t f \Big) \widetilde U_{t+k} g, h \Bigg)
   + \Big( (A U_t f) (\widetilde U_{t+k} g - \widetilde U_t g), h \Big) \Bigg|  \\
& \leq & \|\tfrac{1}{k} (U_{t+k} f - U_t f) - A U_t f\|_2 \,
         \|\widetilde U_{t+k} g\|_\infty \, \|h\|_2
   + \Big| \langle \widetilde U_{t+k} g - \widetilde U_t g, h \, \overline{(A U_t f)} \rangle \Big|
.
\end{eqnarray*}
Since $\lim_{k \to 0} \|\tfrac{1}{k} (U_{t+k} f - U_t f) - A U_t f\|_2 = 0$,
$\sup_{k \in [-1,1]} \|\widetilde U_{t+k} g\|_\infty < \infty$ and
$h \, \overline{(A U_t f)} \in L_1$, it follows from (\ref{epkoop202;6})
that
\[
\lim_{k \to 0} \tfrac{1}{k} \Big( (U_{t+k} f - U_t f) U_{t+k} g, h \Big)
= \Big( (A U_t f) (U_t g) , h\Big)
.  \]
Similarly one proves for the second term in (\ref{epkoop202;5}) that
\[
\lim_{k \to 0} \tfrac{1}{k} \Big( (U_t f) ( U_{t+k} g - U_t g ), h \Big)
= \Big( (U_t f) (A U_t g), h\Big)
.  \]
Hence $t \mapsto (\alpha(t),h)$ is differentiable and
(\ref{epkoop202;7}) is valid.
Thus $\alpha$ is weakly differentiable, with weak derivative
\[
\alpha'(t)
= (A U_t f) (U_t g) + (U_t f) (A U_t g)
.  \]
But $A$ is a derivation on $D(A) \cap L_\infty$. 
Therefore 
\[
\alpha'(t)
= (A U_t f) (U_t g) + (U_t f) (A U_t g)
= A \Big( (U_t f) (U_t g) \Big)
= A (\alpha(t))
\]
for all $t \in \Ri$.
Obviously $\alpha(0) = f \, g$.
The uniqueness of the Cauchy problem yields
$\alpha(t) = e^{t A} (f \, g) = U_t (f \, g)$
for all $t \in \Ri$.

Fix $t \in \Ri$.
Let $g \in D(A) \cap L_\infty$.
Then
\[
U_t (f \, g)
= (U_t f) (U_t g)
= (U_t f) (\widetilde U_t g)
\]
for all $f \in D(A) \cap L_\infty$.
By Lemma~\ref{lkoop205} the space $D(A) \cap L_\infty$ is dense in $L_2$.
Hence it follows by
continuity that $U_t (f \, g) = (U_t f) (\widetilde U_t g)$ is valid for all $f \in L_2$ and
in particular for all $f \in L_2 \cap L_\infty$.
So
\begin{equation}
U_t (f \, g)
= (\widetilde U_t f) (U_t g)
\label{epkoop202;10}
\end{equation}
for all $f \in L_2 \cap L_\infty$ and $g \in D(A) \cap L_\infty$.
Since $D(A) \cap L_\infty$ is dense in $L_2$, it follows
that (\ref{epkoop202;10}) is valid for all
$f,g \in L_2 \cap L_\infty$.
So
\begin{equation}
\widetilde U_t (f \, g)
= (\widetilde U_t f) (\widetilde U_t g)
\label{epkoop202;11}
\end{equation}
for all $f,g \in L_2 \cap L_\infty$.
Let $f \in L_2 \cap L_\infty$ and $h \in L_1$.
Then
\begin{equation}
\langle \widetilde U_t (f \, g), h\rangle
= \langle f \, g , \widehat U_t h\rangle
= \langle g, \overline f \, \widehat U_t h \rangle
\label{epkoop202;2}
\end{equation}
for all $g \in L_\infty$.
Moreover,
\begin{equation}
\langle (\widetilde U_t f) (\widetilde U_t g), h\rangle
= \langle \widetilde U_t g, \overline{\widetilde U_t f} \, h \rangle
= \langle g, \widehat U_t( \overline{\widetilde U_t f} \, h) \rangle
\label{epkoop202;3}
\end{equation}
for all $g \in L_\infty$.
It follows from (\ref{epkoop202;11}), (\ref{epkoop202;2}) and
(\ref{epkoop202;3}) that
\begin{equation}
\langle g, \overline f \, \widehat U_t h \rangle
= \langle g, \widehat U_t( \overline{\widetilde U_t f} \, h) \rangle
\label{epkoop202;4}
\end{equation}
for all $g \in L_2 \cap L_\infty$.
But $L_2 \cap L_\infty$ is w$^*$-dense in $L_\infty$.
So (\ref{epkoop202;4}) is valid for all $g \in L_\infty$.
Using again (\ref{epkoop202;2}) and (\ref{epkoop202;3})
one deduces that
\[
\langle \widetilde U_t (f \, g), h\rangle = \langle (\widetilde U_t f) (\widetilde U_t g), h\rangle
\]
for all $g \in L_\infty$.
This is for all $h \in L_1$.
So (\ref{epkoop202;11}) is valid
for all $f \in L_2 \cap L_\infty$ and $g \in L_\infty$.
Finally, by a similar argument one establishes that (\ref{epkoop202;11}) is valid
for all $f,g \in L_\infty$.
\end{proof}

\begin{thm} \label{tkoop207}
Let $(X,\cb,\mu)$ be a $\sigma$-finite measure space such that
$(X,\cb)$ is a standard Borel space.
Let $U$ be a $C_0$-group on $L_2(X)$
with generator $A$.
Then the following are equivalent.
\begin{tabeleq}
\item \label{tkoop207-1}
For all $t \in \Ri$ there exists a measurable
map $T_t \colon X \to X$ such that $U_t f = f \circ T_t$
for all $f \in L_2(X)$.
\item \label{tkoop207-2}
The space $D(A) \cap L_\infty(X)$ is an algebra
and $A$ is a derivation on $D(A) \cap L_\infty(X)$.
Moreover, there exists a $c > 0$ such that
\[
\|U_t f\|_\infty
\leq c \, \|f\|_\infty
\]
for all $t \in [-1,1]$ and $f \in L_2(X) \cap L_\infty(X)$.
\end{tabeleq}
\end{thm}
\begin{proof}
`\ref{tkoop207-1}$\Rightarrow$\ref{tkoop207-2}'.
This follows from Proposition~\ref{pkoop201}.

`\ref{tkoop207-2}$\Rightarrow$\ref{tkoop207-1}'.
By Proposition~\ref{pkoop202} there exists a unique
one-parameter group $\widetilde U$ on $L_\infty$
which is consistent with $U$.
Moreover,
\begin{equation}
\widetilde U_t (f \, g) = (\widetilde U_t f) (\widetilde U_t g)
\label{etkoop207;1}
\end{equation}
for all $f,g \in L_\infty$ and $t \in \Ri$.
Fix $t \in \Ri$.
Let $\ci = \{ B \in \cb : \mu(B) = 0 \} $.
Then $\ci$ is a $\sigma$-ideal in $\cb$.
Let $B \in \cb$.
Then
$\widetilde U_t \one_B
= \widetilde U_t (\one_B^2)
= (\widetilde U_t \one_B)^2$
by (\ref{etkoop207;1}).
Therefore there exists a $B' \in \cb$ such that
$\widetilde U_t \one_B = \one_{B'}$.
If also $B'' \in \cb$ is such that $\widetilde U_t \one_B = \one_{B''}$,
then $B' \Delta B'' \in \ci$, where $\Delta$ denotes the symmetric difference.
Define $\Phi(B) = B' \Delta \ci \in \cb/\ci$.
Then $\Phi$ is a map from $\cb$ into $\cb/\ci$.

Clearly $\Phi(\emptyset) = \emptyset \Delta \ci$.
Let $B_1,B_2 \in \cb$.
It follows from (\ref{etkoop207;1}) that $\Phi(B_1 \cap B_2) = \Phi(B_1) \wedge \Phi(B_2)$.
Moreover, if $B_1 \cap B_2 = \emptyset$ and $B_1',B_2' \in \cb$ are such that
$\widetilde U_t \one_{B_1} = \one_{B_1'}$ and
$\widetilde U_t \one_{B_2} = \one_{B_2'}$, then
$\one_{B_1' \cap B_2'} = U_t \one_{B_1 \cap B_2} = 0$, so
$\widetilde U_t \one_{B_1 \cup B_2} = \widetilde U_t(\one_{B_1} + \one_{B_2})
= \one_{B_1'} + \one_{B_2'} = \one_{B_1' \cup B_2'}$.
Note that $(\widetilde U_t)^{-1} = \widetilde U_{-t}$ has the same properties
as $\widetilde U_t$.
Hence there exists a $B \in \cb$ such that $\widetilde U_{-t} \one = \one_B$.
Then $\widetilde U_t \one_B = \one$.
Consequently
\[
\one
= \widetilde U_t \one_B
= \widetilde U_t (\one_B \, \one)
= (\widetilde U_t \one_B) (\widetilde U_t \one)
= \one \, \widetilde U_t \one
= \widetilde U_t \one
.  \]
So $\Phi$ is a homomorphism.
Since $\widetilde U_t$ is continuous, it follows that $\widetilde U_t$ is a
$\sigma$-homomorphism of Boolean $\sigma$-algebras.
By \cite{Kec} Theorem~15.9 
there exists a measurable map
$T_t \colon X \to X$ such that
$\Phi(B) = T_t^{-1}(B) \Delta \ci$ for all $B \in \cb$.
So $U_t \one_B = \widetilde U_t \one_B = \one_B \circ T_t$ for all $B \in \cb$
with $\mu(B) < \infty$.
Using the continuity of $U_t$ and the image measure under $T_t$,
one deduces that
$U_t f = f \circ T_t$, first for all $f \in L_1 \cap L_2$ and then
for all $f \in L_2$.
\end{proof}

\begin{cor} \label{ckoop207.5}
Let $(X,\cb,\mu)$ be a standard Borel probability space.
Let $U$ be a $C_0$-group on $L_2(X)$
with generator $A$.
Then the following are equivalent.
\begin{tabeleq}
\item \label{ckoop207.5-1}
For all $t \in \Ri$ there exists a measurable
map $T_t \colon X \to X$ such that $U_t f = f \circ T_t$
for all $f \in L_2(X)$.
\item \label{ckoop207.5-2}
The space $L_\infty(X)$ is invariant under $U$.
Moreover, the space $D(A) \cap L_\infty(X)$ is an algebra
and $A$ is a derivation on $D(A) \cap L_\infty(X)$.
\end{tabeleq}
\end{cor}
\begin{proof}
This is a consequence of Theorems~\ref{tkoop206} and \ref{tkoop207}.
\end{proof}

Note that the map $U_t$ is unitary if and only if the map $T_t$ is measure
preserving in Theorem~\ref{tkoop207}\ref{tkoop207-1}.

\begin{proof}[{\bf Proof of Theorem~\ref{tkoop101}}]
This follows immediately from Corollary~\ref{ckoop207.5}.
\end{proof}

\begin{remark} \label{rkoop209}
Note that in Theorem~\ref{tkoop101} the map $T_t \colon X \to X$ is measure preserving
for all $t \in \Ri$.
Moreover, 
\[
T_{t_1 + T_2} = T_{t_1} \circ T_{t_2} \mbox{ a.e.}
\]
for all $t_1,t_2 \in \Ri$.
Since the one-parameter group $U$ is strongly continuous, it follows from 
\cite{GTW} page~307 that the group $(T_t)_{t \in \Ri}$ enjoys the following 
measurabilty property: there exists a Borel map $F \colon \Ri \times X \to X$
such that for all $t \in \Ri$ one has
\[
F(t,x) = T_t x
\mbox{ for a.e.\ } x \in X
.  \]
Thus $(T_t)_{t \in \Ri}$ is a measurable measure preserving flow.
\end{remark}

\section{Weighted non-singular $C_0$-groups} \label{Skoop3}

Throughout this section let $(X,\cb,\mu)$ be a 
standard Borel probability space.
Let $U$ be a one-parameter group on $L_2(X)$ with $U_0 = I$.
The group $U$ is called {\bf weighted non-singular} if for each $t\in\Ri$ there exist 
a map $T_t\colon X\to X$ and a function $\psi_t\colon X\to\Ci$ such that 
\begin{equation} \label{eSkoop3;1}
U_tf=\psi_t\cdot (f\circ T_t)
\end{equation}
for all $f\in L_2(X)$.
By substituting $f= \one$, we obtain that $\psi_t=U_t\one$ for all $t \in \Ri$, 
in particular, 
$\psi_t$ is measurable.
Moreover, $\psi_0=\one$ and the group property of $U$ implies the 
cocycle identity
\begin{equation} \label{eSkoop3;2}
\psi_{t+t'}=\psi_t\cdot (\psi_{t'}\circ T_t)
\end{equation}
and the group property
\begin{equation} \label{eSkoop3;2.5}
T_{t+t'} = T_t \circ T_{t'}
\mbox{ a.e.}
\end{equation}
for all $t,t'\in\Ri$.
Let $t \in \Ri$.
It follows that $\one=\psi_t\cdot (\psi_{-t} \circ T_t)$, 
whence $\psi_t\neq 0$ a.e.\ and
\begin{equation} \label{eSkoop3;3}
\frac{1}{\psi_t} = \psi_{-t} \circ T_t.
\end{equation}
Therefore 
\[
f\circ T_t=\frac1{\psi_t}\cdot U_t f
\]
for all $f\in L_2(X)$,
so $T_t$ is measurable.
In general, a measurable map $S \colon X \to X$ is called {\bf non-singular}
if $\mu(S^{-1}(A)) = 0$ for all $A \in \cb$ with $\mu(A) = 0$.
Then note that $T_t$ is a non-singular map of $(X,\cb,\mu)$ and that 
the measure $\mu$ and the image measure $T_{t*} \mu$ are equivalent, where
\[
(T_{t*} \mu)(A):=\mu(T_t^{-1} A)
\]
for all $A \in \cb$.
Indeed, if $B \in \cb$ and $\mu(B)=0$, then 
$0=U_t \one_B= \psi_t \cdot (\one_B \circ T_t) = \psi_t\cdot\one_{T^{-1}_t B}$,
hence $(T_{t*} \mu) (B) = \mu(T_t^{-1} B) = 0$.

A weighted non-singular one-parameter group $U$ is called a {\bf weighted Koopman group} if  
$T_t$ is measure-preserving for all $t \in \Ri$.

\begin{lemma} \label{lkoop302}
Let $U$ be a weighted non-singular one-parameter group
given by~{\rm (\ref{eSkoop3;1})}.
Then 
\[
\left\| |\psi_t|^2\cdot \left(\frac{d(T_{t*} \mu)}{d\mu}\circ T_t \right) \right\|_\infty
\leq \|U_t\|^2_{2\to2}
\]
for all $t \in \Ri$.
\end{lemma}
\begin{proof}
If $f \in L_2$, then 
\[
\int|\psi_t|^2 \cdot |f\circ T_t|^2\,d\mu
= \|U_t f\|_2^2
\leq c \, \|f\|_2^2,
\]
where $c=\|U_t\|_{2\to 2}^2$.
Hence
\[
\int|\psi_t|^2 \cdot (f\circ T_t) \, d\mu
\leq c \, \|f\|_1
\]
for all $0 \leq f\in L_1(X,\mu)$.
Equivalently,
\[
\int (|\psi_t|^2\circ T^{-1}_{t}) \cdot f \cdot \frac{d(T_{t*} \mu)}{d\mu} \, d\mu
= \int (|\psi_t|^2 \circ T^{-1}_{t}) \cdot f \, d(T_{t*} \mu)
\leq c \, \|f\|_1
\]
for all $0\leq f\in L_1(X,\mu)$.
Since $(|\psi_t|^2\circ T_{t}^{-1}) \cdot \frac{d(T_{t*} \mu)}{d\mu} \geq 0$,
one deduces that 
\[
\left\| (|\psi_t|^2\circ T_{t}^{-1}) \cdot \frac{d(T_{t*} \mu)}{d\mu} \right\|_\infty
\leq c
\] and the result follows by the non-singularity of $T_t$.
\end{proof}

\begin{lemma} \label{lkoop303}
Let $U$ be a weighted non-singular one-parameter group given by~{\rm (\ref{eSkoop3;1})}.
Then the following are equivalent.
\begin{tabeleq} 
\item \label{lkoop303-1}
The representation $U$ preserves $L_\infty(X)$.
\item \label{lkoop303-2}
$\psi_t=U_t\one \in L_\infty(X)$ for all $t\in \Ri$.
\item \label{lkoop303-3}
$\frac{d(T_{t*} \mu)}{d\mu}\in L_\infty(X)$ for all $t\in \Ri$.
\end{tabeleq}
\end{lemma}
\begin{proof}
`\ref{lkoop303-1}$\Rightarrow$\ref{lkoop303-2}' is trivial and
`\ref{lkoop303-2}$\Rightarrow$\ref{lkoop303-1}' follows from (\ref{eSkoop3;1}) and
the fact that $T_t$ is non-singular for all $t \in \Ri$.

`\ref{lkoop303-2}$\Rightarrow$\ref{lkoop303-3}'.
Lemma~\ref{lkoop302} and (\ref{eSkoop3;3}) imply that 
\[
\left\| \frac{d(T_{t*} \mu)}{d\mu}\circ T_t \right\|_\infty
\leq \|U_t\|^2_{2\to2} \, \|\psi_{-t}\|_\infty^2
< \infty
\]
for all $t \in \Ri$.

`\ref{lkoop303-3}$\Rightarrow$\ref{lkoop303-2}'.
Since $T_t \circ T_{-t} = I$ a.e., it follows that 
\[
\Big( \frac{d(T_{t*}\mu)}{d\mu} \circ T_t \Big) 
    \cdot \frac{d(T_{-t*}\mu)}{d\mu}
= \one
\]
for all $t \in \Ri$.
Then the claim is a consequence of Lemma~\ref{lkoop302}.
\end{proof}

\begin{remark} \label{rkoop304}
Let $U$ be a $C_0$-group which is weighted Koopman and unitary.
Let $t \in \Ri$.
Then 
\[
\int |f|^2\circ T_t \, d\mu
= \int |f|^2\,d\mu
= \|U_t f\|^2_2
= \int|\psi_t|^2 \cdot (|f|^2 \circ T_t) \, d\mu
\]
for all $f \in L_2(X)$.
Hence $\int (|\psi_t|^2-1) \cdot (|f|^2\circ T_t) \, d\mu = 0$
for all $f \in L_2(X)$
and therefore
$|\psi_t|=1$ a.e.
\end{remark}

There are many one-parameter $C_0$-groups which preserve $L_\infty(X)$,
but which are not weighted non-singular.

\begin{exam} \label{xkoop304.5}
Let $B \in \cb$ be such that $\mu(B) \neq 0 \neq \mu(X \setminus B)$.
Define $A \colon L_2(X) \to L_2(X)$ by $Af = (f,\one_B) \, \one_{X \setminus B}$.
Then $A$ is bounded, so it generates a $C_0$-group $U$.
Since $A^2 = 0$, one deduces that $U_t = I + t \, A$ for all $t \in \Ri$.
Hence obviously $U$ leaves $L_\infty(X)$ invariant.
Now choose $t = - \mu(B)^{-1}$.
Then 
\[
U_t \one 
= \one + t \, (\one,\one_B) \, \one_{X \setminus B}
= \one + t \, \mu(B) \, \one_{X \setminus B}
= \one - \one_{X \setminus B}
= \one_B
.  \]
Since $\mu( \{ x \in X : (U_t \one)(x) = 0 \} ) = \mu(X \setminus B) > 0$, 
the group $U$ is not weighted non-singular by (\ref{eSkoop3;3}).
\end{exam}

We next consider weighted non-singular one-parameter groups which 
preserve $L_\infty(X)$.

\begin{lemma} \label{lkoop305}
Let $U$ be a weighted non-singular one-parameter group given by~{\rm (\ref{eSkoop3;1})}.
Assume that $U$ preserves $L_\infty(X)$.
Then $f \circ T_t \in L_2(X)$ for all $f\in L_2(X)$ and $t\in\Ri$.
Define $V_t \colon L_2(X) \to L_2(X)$ by 
\[
V_t f = f \circ T_t 
.  \]
Then one has the following.
\begin{tabel}
\item  \label{lkoop305-1}
$(V_t)_{t \in \Ri}$ is a one-parameter group on $L_2(X)$.
\item  \label{lkoop305-2}
If $U$ is a $C_0$-group, then also $(V_t)_{t \in \Ri}$ is a $C_0$-group.
\end{tabel}
\end{lemma}
\begin{proof} 
Note that (\ref{eSkoop3;1}) and (\ref{eSkoop3;3}) imply that 
\[
V_t f 
= f \circ T_t 
= (\psi_{-t} \circ T_t) \, U_t f
\in L_2
\]
for all $t \in \Ri$ and $f \in L_2$.
Then Statement~\ref{lkoop305-1} is a consequence of (\ref{eSkoop3;2.5}).

`\ref{lkoop305-2}'.
By Theorem~\ref{tkoop206} there exist $M \geq 1$ and $\omega \geq 0$
such that $\|U_t f\|_\infty \leq M \, e^{\omega |t|} \, \|f\|_\infty$ 
for all $t \in \Ri$ and $f \in L_\infty$.

Fix $f\in L_\infty$.
Let $t\in(0,1)$.
Then (\ref{eSkoop3;3}) gives
\begin{eqnarray}
V_t f - f
& = & \frac{1}{U_t \one} \Big( (U_t f - f) + (\one - U_t \one) f \Big)  \nonumber \\
& = & \Big( (U_{-t} \one) \circ T_t \Big) \Big( (U_t f - f) + (\one - U_t \one) f \Big)
.
\label{elkoop305;3}
\end{eqnarray}
Therefore 
\begin{eqnarray}
\|V_t f - f\|_2
& \leq & \|(U_{-t} \one) \circ T_t\|_\infty 
   \Big( \|U_t f - f\|_2 + \|\one - U_t \one\|_2 \, \|f\|_\infty \Big)  \nonumber   \\
& \leq & M \, e^\omega \, \Big( \|U_t f - f\|_2 + \|\one - U_t \one\|_2 \, \|f\|_\infty \Big)
\label{elkoop305;2}
\end{eqnarray}
and $\lim_{t \downarrow 0} V_t f = f$.
Then the result follows since $L_\infty$ is dense in $L_2$.
\end{proof}

\begin{prop} \label{pkoop306}
Let $(X,\cb,\mu)$ be a standard Borel probability space.
Let $U$ be a $C_0$-group on $L_2(X)$ preserving $L_\infty(X)$.
Then the following are equivalent.
\begin{tabeleq}
\item \label{pkoop306-1}
The representation $U$ is weighted non-singular.
\item \label{pkoop306-2}
For all $t \in \Ri$ one has $U_t \one \neq 0$ a.e.\ and $\frac{1}{U_t \one} \in L_\infty(X)$.
Moreover, $V = (V_t)_{t \in \Ri}$ is a $C_0$-group on $L_2(X)$, where
\begin{equation}
V_t f = \frac{1}{U_t \one} \, U_t f
\label{epkoop306;1}
\end{equation}
for all $t \in \Ri$.
In addition $D(B) \cap L_\infty(X)$ is an algebra and $B$ is a derivation on 
$D(B) \cap L_\infty(X)$, where $B$ is the generator of $V$.
\end{tabeleq}
\end{prop}
\begin{proof} 
`\ref{pkoop306-1}$\Rightarrow$\ref{pkoop306-2}'.
This follows from (\ref{eSkoop3;3}), Lemma~\ref{lkoop305}\ref{lkoop305-2}
and Proposition~\ref{pkoop201}.

`\ref{pkoop306-2}$\Rightarrow$\ref{pkoop306-1}'.
It follows from (\ref{epkoop306;1}) that $V$ leaves $L_\infty$ invariant.
Then apply Corollary~\ref{ckoop207.5} to $V$ and the result follows from (\ref{epkoop306;1}).
\end{proof}

\begin{cor}\label{ckoop309}
Let $(X,\cb,\mu)$ be a standard Borel probability space.
Let $U$ be a unitary $C_0$-group on $L_2(X)$ preserving $L_\infty(X)$.
Then the following are equivalent.
\begin{tabeleq}
\item \label{ckoop309-1}
The group $U$ is a weighted Koopman group.
\item \label{ckoop309-2}
For all $t \in \Ri$ one has $|U_t \one| = 1$ a.e.
Moreover, $V = (V_t)_{t \in \Ri}$ is a unitary $C_0$-group on $L_2(X)$, where
\[
V_t f = \overline{U_t \one} \cdot U_t f
\]
for all $t \in \Ri$.
In addition $D(B) \cap L_\infty(X)$ is an algebra and $B$ is a derivation on 
$D(B) \cap L_\infty(X)$, where $B$ is the generator of $V$.
\end{tabeleq}
\end{cor}

In order to obtain a relationship between the generators of the two $C_0$-groups in 
Lemma~\ref{lkoop305}\ref{lkoop305-2}, we need the following observation.

\begin{lemma} \label{lkoop307} 
Let $U$ be a weighted non-singular one-parameter $C_0$-group.
Let $V = (V_t)_{t \in \Ri}$ be the group on $L_2(X)$ as in Lemma~{\rm \ref{lkoop305}}.
Then 
\[
\lim_{t\to0} \|V_t(U_{-t}\one) \cdot g - g\|_2
= 0
\]
for all $g\in L_\infty(X)$.
\end{lemma}
\begin{proof} 
It follows from Lemma~\ref{lkoop305}\ref{lkoop305-2} that $V$ is a $C_0$-group.
Hence $\sup_{t \in [-1,1]} \|V_t\|_{2 \to 2} < \infty$.
Let $t \in (-1,1)$.
Then
\begin{eqnarray*}
\|V_t(U_{-t}\one)\cdot g -g\|_2
& = & \|\Big( V_t(U_{-t}\one-\one)+ V_t\one-\one \Big) g\|_2  \\
& \leq & \Big( \|V_t\|_{2\to2} \|U_{-t}\one-\one\|_2 + \|V_t\one-\one\|_2 \Big) 
    \|g\|_\infty
\end{eqnarray*}
and the result follows.
\end{proof}

\begin{lemma}\label{lkoop308}
Let $U$ be a weighted non-singular one-parameter $C_0$-group.
Assume that $U$ preserves $L_\infty(X)$.
Let $V = (V_t)_{t \in \Ri}$ be the $C_0$-group on $L_2(X)$ as in Lemma~{\rm \ref{lkoop305}}.
Denote by $A$ and $B$ the generators of $U$ and $V$, respectively.
Assume that
\[
\one\in D(A).
\]
Then $D(A)\cap L_\infty(X)=D(B)\cap L_\infty(X)$ and $Bf=Af-f\cdot A\one$ 
for each $f\in D(A)\cap L_\infty(X)$.
\end{lemma}
\begin{proof}
Let $f\in D(A)\cap L_\infty$.
Since $\one \in D(A)$ 
it follows from (\ref{elkoop305;2}) that there exists a $c > 0$ such that 
$\|V_tf-f\|_2\leq c \, t$
for all $t\in (0,1)$.
Therefore $f\in D(B)$ by \cite{EN} Corollary~II.5.21.
Hence $D(A)\cap L_\infty\subset D(B)\cap L_\infty$.
Let $g\in L_\infty$.
Then (\ref{elkoop305;3}) gives
\[
\frac{1}{t} (V_t f-f,g)
= \Big( \frac{1}{t} (U_t f- f) - f \cdot \frac1t(U_t\one-\one) ,
       \overline{(U_{-t}\one) \circ T_t} \cdot g \Big)
\]
for all $t \in (0,1)$.
Now take the limit $t \to 0$ and use Lemma~\ref{lkoop307}.
It follows that 
\[
(Bf,g)=(Af-f\cdot A\one,g).
\]
Therefore $Bf=Af-(A\one)\cdot f$.

Conversely, let $f\in D(B)\cap L_\infty$.
Then $U_tf-f=(U_t\one)(V_tf-f)+(U_t\one-\one)f$ for all $t \in \Ri$.
The bounds (\ref{etkoop206;10}) of Theorem~\ref{tkoop206} imply that there exists
a $c > 0$ such that 
\[
\|U_tf-f\|_2\leq\|U_t\one\|_\infty\|V_tf-f\|_2+
\|U_t\one-\one\|_2\|f\|_\infty
\leq c \, |t|
\]
for all $t \in (0,1)$.
Hence $f \in D(A)$ as before.
\end{proof}

We can now prove the main theorem of this section.

\begin{thm} \label{tkoop310}
Let $(X,\cb,\mu)$ be a standard Borel probability space.
Let $U$ be a $C_0$-group on $L_2(X)$ preserving $L_\infty(X)$.
 Assume that $\one\in D(A)$ with $A\one\in L_\infty(X)$, 
where $A$ is the generator of $U$.
Then the following are equivalent.
\begin{tabeleq}
\item \label{tkoop310-1}
The representation $U$ is weighted non-singular.
\item \label{tkoop310-2}
The space $D(A)\cap L_\infty(X)$  is an algebra and
$A - (A\one) I$ is a derivation on $D(A)\cap L_\infty(X)$.
\end{tabeleq}
\end{thm}
\begin{proof}
`\ref{tkoop310-1}$\Rightarrow$\ref{tkoop310-2}'.
This follows from Proposition~\ref{pkoop306} and Lemma~\ref{lkoop308}.
Note that this implication does not require the assumption $A\one\in L_\infty$.

`\ref{tkoop310-2}$\Rightarrow$\ref{tkoop310-1}'.
Consider first $U^*$, which is a $C_0$-group on $L_2$ whose generator is $A^*$.
By Theorem~\ref{tkoop206}\ref{tkoop206-2}$\Rightarrow$\ref{tkoop206-1}
the one-parameter group $U^*$ extends consistently to a 
$C_0$-group $\widehat{U}$ on $L_1$.
Denote by $\widehat{A}$ the generator of this group.

Since $(A\one) I$ is  a bounded operator
the operator $A-(A\one)I$ generates a $C_0$-group $V$ on $L_2$ 
by perturbation theory \cite{EN}, Theorem~III.1.3. 
Then $A^*-\overline{(A\one)}I$ is the generator of $V^*$.
Moreover, again by perturbation theory, $\widehat{A}-\overline{(A\one)}I$ 
is the generator of a $C_0$-group $\widehat{V}$ on $L_1$.
Let $t\in\Ri$.
The Trotter--Kato formula \cite{EN} Exercise~III.5.11(1) gives
\[
V_t^*
=\lim_{n\to\infty}\left(e^{-\frac tn\overline{(A\one)} I} \, U^*_{\frac tn}\right)^n
   \;\;\mbox{strongly in } \cl(L_2)
\]
and
\[
\widehat{V}_t
=\lim_{n\to\infty}\left(e^{-\frac tn\overline{(A\one)} I} \, \widehat{U}_{\frac tn}\right)^n
\;\;\mbox{strongly in } \cl(L_1).
\]
Let $f\in L_2$.
Then $f \in L_1$ and since $U^*$ and $\widehat U$ are consistent one deduces that 
\[
\left(e^{-\frac tn\overline{(A\one)} I} \, U^*_{\frac tn}\right)^n f
=\left(e^{-\frac tn\overline{(A\one)} I} \, \widehat{U}_{\frac tn}\right)^nf \;\mbox{a.e.}
\]
for all $n\in\Ni$.
Hence $V^*_t f = \widehat{V}_t f$ a.e. and $V^*$ and $\widehat{V}$ are consistent.
By Theorem~\ref{tkoop206}\ref{tkoop206-1}$\Rightarrow$\ref{tkoop206-2},
applied with $S=V^*$, 
it follows that $V$ leaves the space $L_\infty$ invariant.
By Theorem~\ref{tkoop207} it follows that for all $t \in \Ri$ there exists a 
non-singular measurable map $T_t \colon X\to X$ such that 
$V_tf=f\circ T_t$ for all $f \in L_2$.

Note that  
\[
\Big( V_t \circ e^{s(A\one)I} \Big) f
= V_t (e^{s(A\one)} f)
= (e^{s(A\one)} f) \circ T_t
= \Big( e^{s ((A\one)\circ T_s) I} \circ V_t \Big) f
\]
for all $t,s\in \Ri$ and $f \in L_2$.
Iteration gives
\begin{equation}
\left( V_{\frac{t}{n}} \circ e^{\frac{t}{n} (A\one) I} \right)^n
= e^{\frac{t}{n} \big((A\one)\circ T_{\frac{t}{n}}+\ldots+ (A\one)\circ T_{\frac{nt}{n}} \big) I }
    \circ \Big( V_{\frac{t}{n}} \Big)^n 
= e^{\frac{t}{n} \big((A\one)\circ T_{\frac{t}{n}}+\ldots+ (A\one)\circ T_{\frac{nt}{n}} \big) I }
    \circ V_t
\label{etkoop310;1}
\end{equation}
for all $t \in \Ri$ and $n \in \Ni$.
Since $A = (A - (A\one) I) + (A\one) I$, one can consider the generator 
of the $C_0$-group $U$ as a perturbation of the generator of the $C_0$-group $V$.
Then the Trotter--Kato formula gives 
\[
U_t = \lim_{n \to \infty} \left( V_{\frac{t}{n}} \circ e^{\frac{t}{n} (A\one) I} \right)^n
\]
strongly in $\cl(L_2)$.
Hence (\ref{etkoop310;1}) gives $U_t = \psi_t \cdot V_t$ for all $t \in \Ri$, where
\[
\psi_t = e^{ \int_0^t (A\one) \circ T_r \, dr}
\in L_\infty
.  \]
This completes the proof.
\end{proof}

Clearly Theorem~\ref{tkoop311} is a consequence of Theorem~\ref{tkoop310}.

\medskip

The condition $\one \in D(A)$ is not satisfied in general.
We give a wide class of examples.

\begin{exam} \label{xkoop312}
Let $V = (V_t)_{t \in \Ri}$ be a unitary $C_0$-group on $L_2(X)$ given by 
a measure preserving flow $T=(T_t)_{t\in\Ri}$ which is ergodic. 
So $V_t f = f \circ T_t$ for all $t \in \Ri$ and $f \in L_2(X)$ and 
the only $f \in L_2(X)$ which are invariant under $V_t$ for all $t \in \Ri$
are the constants.
We will now show that for all $t \in \Ri$ 
we can find a measurable $\psi_t\colon X\to \Ri$, 
bounded and bounded away from zero, such that $U = (U_t)_{t \in \Ri}$ is 
a continuous $C_0$-group on $L_2(X)$ for which $\one\notin D(A)$,
where 
\[
U_tf= \psi_t\cdot (f\circ T_t)
\]
for all $t \in \Ri$.

Indeed, by Ambrose--Kakutani theorem, see for example \cite{CFS} Theorem~11.2.1, 
we can represent $T$ 
as a special flow over an ergodic automorphism $S$ of a standard Borel probability 
space $(Y,\cc,\rho)$, i.e.\ there exist $F\colon Y \to \Ri$ and $c > 0$
such that $F>c$,
$\int_YF\,d\rho < \infty$ and
\[
X=Y^F:=\{(y,s) \in Y\times \Ri : 0\leq s\leq F(y) \}.
\]
On $Y^F$ we consider the restriction of the product measurable structure from 
$Y\times\Ri$ together with $\rho^F:= (\rho\otimes \Leb_{\Ri})|_{Y^F}$.
The flow $T$ acts as $S^F=(S_t^F)_{t\in\Ri}$, where under the action of 
$S^F_t$ (with $t>0$) a point $(y,r)$ moves up vertically with unit speed until 
it hits the point $(y,f(y))$ which is identified with $(Sy,0)$ and this movement 
is continued until time $t$.
In this way we obtain
a unitary $C_0$-group $V=(V_t)_{t\in\Ri}$, where $V_tf=f\circ S^F_t$ on 
$L_2(Y^F,\rho^F)$.

Let $a,b \in \Ri$ be such that $0<a<b<c$ and consider the strip $H:=Y\times[a,b]$.
Then $H\subset Y^F$ and $\rho^F(H)=b-a$.
For each $t\in\Ri$ with  $|t| < a \wedge (c-b) \wedge (b-a)$ one has
\begin{equation} \label{eSkoop3;10}
\rho^F(H\triangle S^F_t(H))=2|t|.
\end{equation}
We claim that $g:=\one_H\notin D(B)$, where $B$ is the generator of $V$.
Indeed, for all $t\in\Ri$ with  $|t| < a \wedge (c-b) \wedge (b-a)$, 
it follows from (\ref{eSkoop3;10}) that 
\[
\|g-g\circ S^F_t\|_2
=\Big(\int_Y |\one_H-\one_H\circ S^F_t|^2\,d\rho^F\Big)^{1/2}
=\left(\rho^F(H\triangle S^F_t(H))\right)^{1/2}
=\sqrt{2|t|}.
\]
Therefore there is no constant $\kappa>0$ such that
$\|g-g\circ S^F_t\|_2\leq \kappa \, |t|$ for all sufficiently small $|t|>0$ and hence 
$g\notin D(B)$.

Let $\theta:=g + \one$.
Then $\theta\notin D(B)$ and $\theta,\frac{1}{\theta} \in L_\infty(Y^F)$.
Set $\psi_t:=\frac{\theta}{\theta\circ S^F_t}$ for all $t\in\Ri$.
Then $(\psi_t)_{t\in\Ri}$ satisfies the cocycle identity~(\ref{eSkoop3;2}) and by setting
\[
U_t f
= \psi_t \cdot (f\circ S^F_t),
\]
we obtain a $C_0$-group $U = (U_t)_{t \in \Ri}$ on $L_2(Y^F)$.
Now
\[
\frac{1}{t} \, (U_t\one - \one)
=\frac{1}{t} \, (\theta-\theta\circ S^F_t) \cdot \frac{1}{\theta \circ S^F_t}
=\frac{1}{t} \, (\theta- V_t \theta) \cdot \frac{1}{\theta \circ S^F_t}
\]
and since $V$ is a $C_0$-group and $\theta\notin D(B)$, we must have $\one\notin D(A)$,
where $A$ is the generator of $U$.
\end{exam}

\begin{remark} \label{rkoop313}
By considering the function $\xi = (-1)^{\one_H} = \one_{X \setminus H} - \one_H$, 
we obtain a measurable function for which $\xi\notin D(B)$ taking values in 
$\{ -1,1 \}$, and if we set $\psi_t:= \frac{\xi}{\xi\circ S^F_t}$, then the corresponding group 
$U$ is weighted Koopman for which $\one\notin D(A)$.
\end{remark}

Even if $\one \in D(A)$, then in general $A \one \not\in L_\infty(X)$.
An example is as follows.

\begin{exam} \label{xkoop314}
Let $\Ti = \{ z \in \Ci : |z| = 1 \} $ be the torus with normalized Haar measure.
Let 
\[
E = \{ \eta \in L_2(\Ti) : \int \eta = 0 \} 
.  \]
Then $E$ is a closed subspace of $L_2(\Ti)$.
We provide $E$ with the norm of $L_2(\Ti)$.
For all $\eta \in L_2(\Ti)$ define $\tilde \eta \in L_{2,\loc}(\Ri)$ by 
$\tilde \eta(x) = \eta(e^{ix})$.

Fix $\zeta \in E$.
For all $t \in \Ri$ define $\varphi_t \in C(\Ti)$ by 
\[
\varphi_t(e^{ix}) = \int_x^{x+t} \tilde \zeta
.  \]
Note that $\varphi_t$ is well defined.
Since $\int_\Ti \zeta = 0$ one deduces that $\|\varphi_t\|_\infty \leq 2 \pi \|\zeta\|_1$.
If $s,t \in \Ri$ then 
\[
\tilde \varphi_{t+s}(x)
= \int_x^{x+t} \tilde \zeta + \int_{x+t}^{x+t+s} \tilde \zeta
= \tilde \varphi_t(x) + \tilde \varphi_s(x+t)
\]
for all $x \in \Ri$.
For all $t \in \Ri$ define $\psi_t \in C(\Ti)$ by 
\[
\psi_t = e^{\varphi_t}
\]
and define $U_t \colon L_2(\Ti) \to L_2(\Ti)$ by 
\[
(U_t f)(z) = \psi_t(z) \, f(e^{it} \, z)
.  \]
It is easy to verify that $U_t L_\infty(\Ti) \subset L_\infty(\Ti)$
for all $t \in \Ri$ and that $U = (U_t)_{t \in \Ri}$ is a 
$C_0$-group.
Let $A$ be the generator of $U$.
Clearly $\psi_t = U_t \one$ for all $t \in \Ri$.
Up to now everything also works if $\zeta \in L_1(\Ti)$ with $\int \zeta = 0$, 
but from now on we use that $\zeta \in L_2(\Ti)$.
We shall prove that $\one \in D(A)$ and $A \one = \zeta$.

Let $t \in (0,1)$.
Then 
\begin{eqnarray}
\Big| \frac{1}{t} (U_t \one - \one) - \zeta \Big|
& \leq & \Big| \frac{e^{\varphi_t} - \one - \varphi_t}{t} \Big|
    + \Big| \frac{1}{t} \, \varphi_t - \zeta \Big|  \nonumber  \\
& \leq & \frac{1}{t} |\varphi_t|^2 \, e^{|\varphi_t|}
    + \Big| \frac{1}{t} \, \varphi_t - \zeta \Big|  \nonumber   \\
& \leq & \frac{1}{t} |\varphi_t|^2 \, e^{2 \pi \, \|\zeta\|_1}
    + \Big| \frac{1}{t} \, \varphi_t - \zeta \Big| 
. \label{exkoop314;1} 
\end{eqnarray}
We estimate the terms in (\ref{exkoop314;1}) separately in $L_2(\Ti)$
in the limit $t \downarrow 0$.

We start with the second term.
For all $t \in (0,1)$ define $F_t \colon E \to C(\Ti)$ by 
\[
(F_t \eta)(e^{ix}) 
= \frac{1}{t} \, \int_x^{x+t} \tilde \eta
.  \]
Note that $F_t(\zeta) = \frac{1}{t} \, \varphi_t$.
Let $\eta \in E$ and $\tau \in L_2(\Ti)$.
Then Fubini and Cauchy--Schwarz give
\begin{eqnarray*}
|(F_t(\eta), \tau)_{L_2(\Ti)}|
& = & \frac{1}{t} \Big| \int_0^{2\pi} \int_x^{x+t} \tilde \eta(s) \, ds 
             \, \overline{\tilde \tau(x)} \, dx \Big|  \\
& = & \frac{1}{t} \Big| \int_0^{2\pi} \int_0^t \tilde \eta(x+s) \, ds 
             \, \overline{\tilde \tau(x)} \, dx \Big|  \\
& = & \frac{1}{t} \Big| \int_0^t \int_0^{2\pi} \tilde \eta(x+s) 
             \, \overline{\tilde \tau(x)} \, dx \, ds \Big| \\
& \leq & \frac{1}{t} \int_0^t 2 \pi \, \|\eta\|_{L_2(\Ti)} \, \|\tau\|_{L_2(\Ti)} \, ds  \\
& = & 2 \pi \, \|\eta\|_{L_2(\Ti)} \, \|\tau\|_{L_2(\Ti)}
.
\end{eqnarray*}
So $\|F_t(\eta)\|_{L_2(\Ti)} \leq 2 \pi \, \|\eta\|_{L_2(\Ti)}$ and 
the set $ \{ F_t : t \in (0,1) \} $ is bounded in $\cl(E,L_2(\Ti))$.
Clearly $\lim_{t \downarrow 0} F_t(\eta) = \eta$ in $L_2(\Ti)$ for all $\eta \in C(\Ti)$.
Since $E \cap C(\Ti)$ is dense in $E$, it follows that 
$\lim_{t \downarrow 0} F_t(\eta) = \eta$ in $L_2(\Ti)$ for all $\eta \in E$.
In particular for $\zeta$ one deduces that 
\begin{equation}
\lim_{t \downarrow 0} \Big| \frac{1}{t} \, \varphi_t - \zeta \Big| = 0
\label{exkoop314;2}
\end{equation}
in $L_2(\Ti)$.
This settles the second term in (\ref{exkoop314;1}).

Now we consider the first term in (\ref{exkoop314;1}).
We shall show that $\lim_{t \downarrow 0} \frac{1}{t} |\varphi_t|^2 = 0$ in $L_2(\Ti)$.
If $t \in (0,1)$, then 
\[
|\varphi_t(e^{ix})|
= \Big| \int_x^{x+t} \tilde \zeta \Big|
\leq  \sqrt{2 \pi \, t} \, \|\zeta\|_2
\]
for all $x \in \Ri$ by the Cauchy--Schwarz inequality.
So $\| \frac{1}{t} |\varphi_t|^2 \|_\infty \leq 2 \pi \, \|\zeta\|_2^2$
for all $t \in (0,1)$.
Let $t_1,t_2,\ldots \in (0,1)$ and assume that $\lim_{n \to \infty} t_n = 0$.
Then passing to a subsequence if necessary, it follows from (\ref{exkoop314;2})
that $\lim_{n \to \infty} \frac{1}{t_n} \, \varphi_{t_n}(z) = \zeta(z)$ 
for a.e.\ $z \in \Ti$.
Hence 
\[
\lim_{n \to \infty} \frac{1}{t_n} \, |\varphi_{t_n}(z)|^2
= \lim_{n \to \infty} t_n \, \Big| \frac{1}{t_n} \, \varphi_{t_n}(z) \Big|^2
= 0
\]
for a.e.\ $z \in \Ti$.
Then the bounded convergence theorem of Lebesgue gives
$\lim_{n \to \infty} \frac{1}{t_n} |\varphi_{t_n}|^2 = 0$ in $L_2(\Ti)$.
Hence $\lim_{t \downarrow 0} \frac{1}{t} |\varphi_t|^2 = 0$ in $L_2(\Ti)$.

Combining the two estimates
it follows from (\ref{exkoop314;1}) that $\one \in D(A)$ and $A \one = \zeta$.
Finally, if one chooses $\zeta \in E$ such that $\zeta \not\in L_\infty(\Ti)$, 
then $A \one \not\in L_\infty(\Ti)$.
\end{exam}

\section{Cocycles} \label{Skoop4}

In the previous section we started with a group $U$ on $L_2(X)$ and 
in case $U$ was weighted non-singular as in (\ref{eSkoop3;1}), we defined
the representation $V$ given by $V_t f = f \circ T_t$.
In that case $U_t = \psi_t \, V_t$.
In this section we reverse the order.
We start with a representation of the form $V_t f = f \circ T_t$
and wish to construct as general as possible a representation $U$ 
of the form (\ref{eSkoop3;1}), that is  $U_t = \psi_t \, V_t$ for all $t \in \Ri$.

Throughout this section let $(X,\cb,\mu)$ be a 
standard Borel probability space.
For all $t\in\Ri$ let $T_t \colon X\to X$ be a measurable map
such that $V = (V_t)_{t \in \Ri}$ is a $C_0$-group on $L_2(X)$, 
where $V_t f:=f\circ T_t$ for all $t\in\Ri$.
Let $B$ be the generator of $V$.Ê 

We need a few definitions.
A mapÊ $\psi \colon \Ri \to L_\infty(X)$ is said to be a 
{\bf cocycle (over $V$)} if 
\begin{equation}
\psi_{t+t'} = \psi_t \cdot (\psi_{t'} \circ T_t)
\label{eSkoop401;20}
\end{equation}
for all $t,t'\in\Ri$, where we write for simplicity $\psi_t = \psi(t)$
for all $t \in \Ri$.
Note that $\psi = 0$ is a cocycle over $V$.
Suppose that $\psi$ is a cocycle.
For all $t \in \Ri$ define $U_t = \psi_t \, V_t \in \cl(L_2(X))$.
Clearly $\|U_t\|_{2 \to 2} \leq \|\psi_t\|_\infty \|V_t\|_{2\to2}$.
If $t,t'\in\Ri$ then
\[
U_{t+t'}f
= \Big( \psi_t\cdot(\psi_{t'}\circ T_t) \Big) V_{t+t'}f
= U_t (U_{t'}f)
= (U_t \circ U_{t'}) f
\]
for all $f\in L_2(X)$, 
so $U = (U_t)_{t\in\Ri}$ is a one-parameter group on $L_2(X)$, 
which leaves $L_\infty(X)$ invariant.
We call $U$ the {\bf  one-parameter group associated with $\psi$}.
Possibly $U_0 = 0$.
With a continuity condition this is not the case.

\begin{lemma} \label{lkoop400.3}
If $\lim_{t \to 0} \|\psi_t - \one\|_1 = 0$, then $\psi_0 = \one$ a.e.\ and $U_0 = I$.
\end{lemma}
\begin{proof}
Let $B \in \cb$ and suppose that $\psi_0|_B = 0$ a.e.
Then $\psi_t|_B = 0$ a.e.\ by (\ref{eSkoop401;20}).
Since $\lim_{t \to 0} \|\psi_t - \one\|_1 = 0$, one deduces that $\mu(B) = 0$.
So $\psi_0 \neq 0$ a.e.
In addition, (\ref{eSkoop401;20}) gives $\psi_0 = \psi_{0+0} = \psi_0^2$.
Hence $\psi_0 = \one$ a.e.
\end{proof}

The cocycle $\psi$ is called a $C_0$-{\bf cocycle (over $V$)} if 
$U$ is a $C_0$-group on $L_2(X)$.
If $\theta\in L_\infty(X)$ is such that
$\theta \neq 0$ a.e., and $\frac{1}{\theta} \in L_\infty(X)$, then 
it is easy to verify that $t \mapsto \frac{\theta\circ T_t}{\theta}$ is a 
cocycle.
A cocycle $\psi$ is called a {\bf coboundary} if there exists a $\theta\in L_\infty(X)$
such that
$\theta \neq 0$ a.e., $\frac{1}{\theta} \in L_\infty(X)$ and 
\[
\psi_t=\frac{\theta\circ T_t}{\theta}
\] 
for all $t\in\Ri$.
The function $\theta$ is called a {\bf transfer function} of the coboundary.
If, in addition, $\theta\in D(B)$ and $B \theta \in L_\infty(X)$, then $\psi$ is called a 
{\bf coboundary with an $L_\infty$-differentiable transfer function}.

If $\psi$ is a cocycle and $\zeta \in L_2(X)$, then $\zeta$ is called the
{\bf derivative} of $\psi$ if
$\lim_{t\to0} \frac{1}{t} (\psi_t - \one) = \zeta$ in $L_2(X)$.
We say that a cocycle $\psi$ is {\bf differentiable} if there exists an $\zeta \in L_2(X)$
such that $\zeta$ is the derivative of $\psi$.

We start with a characterisation of $C_0$-cocycles.

\begin{prop} \label{pkoop400.5}
Let $\psi \colon \Ri\to L_\infty(X)$ be a cocycle over $V$.
Then the following are equivalent.
\begin{tabeleq}
\item \label{pkoop400.5-1}
$\psi$ is a $C_0$-cocycle.
\item \label{pkoop400.5-2}
$\lim_{t \to 0} \|\psi_t - \one\|_2 = 0$.
\end{tabeleq}
\end{prop}
\begin{proof}
Let $U$ be the one-parameter group associated with $\psi$.
Since $\psi_t = U_t \one$ for all $t \in \Ri$, the implication 
\ref{pkoop400.5-1}$\Rightarrow$\ref{pkoop400.5-2} is trivial.
So it remains to prove the converse.

Because $\lim_{t \to 0} \|\psi_t - \one\|_1 = 0$ by \ref{pkoop400.5-2}, 
it follows from Lemma~\ref{lkoop400.3} that $\psi_0 = \one$ a.e.
Clearly $\|U_t\|_{2\to2}\leq\|\psi_t\|_\infty\|V_t\|_{2\to2}$ for all $t \in \Ri$.
If $f\in L_\infty$ then
$\|U_t f\|_\infty\leq \|\psi_t\|_\infty\|f\|_\infty < \infty$.
Hence the operator $\widetilde U_t:=U_t|_{L_\infty} \colon L_\infty\to L_\infty$ 
is bounded.
Obviously $(\widetilde U_t)_{t\in\Ri}$ is a one-parameter
group on $L_\infty$.

Let $f\in L_\infty$.
Then
\[
U_tf-f
=(\psi_t-\one)V_tf+V_tf-f
\]
for all $t\in\Ri$, so
\[
\|U_tf-f\|_2
\leq \|\psi_t-\one\|_2 \|V_tf\|_\infty + \|V_tf-f\|_2
= \|f\|_\infty\|\psi_t-\one\|_2+\|V_tf-f\|_2
\]
and therefore 
\begin{equation} \label{epkoop401;1}
\lim_{t\to 0} \|U_tf-f\|_2 = 0
\end{equation}
by assumption.
Fix $t_0\in\Ri$.
Then
\[
\|U_{t_0+t}f-U_{t_0}f\|_2\leq\|U_{t_0}\|_{2\to2}\|U_tf-f\|_2
\]
for all $t\in\Ri$.
Hence, $\lim_{t\to t_0} U_tf=U_{t_0}f$ in $L_2$.
So $t \mapsto U_t f$ is continuous from $\Ri$ into $L_2$.
Hence the map $t \mapsto |(U_tf,g)|$ from $\Ri$ into $\Ri$ is continuous
for all $g \in L_2$.

Let $f\in L_\infty$ and $t\in\Ri$.
Then
\begin{eqnarray*}
\|\widetilde U_tf\|_\infty
& =& \sup\{|\langle \widetilde U_tf,g\rangle| : g\in L_1 \mbox{ and } \|g\|_1\leq1\}  \\
& = & \sup\{|\langle \widetilde U_tf,g\rangle|: g\in L_2 \mbox{ and } \|g\|_1\leq1\}  \\
& =& \sup\{|( U_tf,g)| : g\in L_2 \mbox{ and } \|g\|_1\leq1\}.
\end{eqnarray*}
Since the map $t\mapsto |(U_tf,g)|$ is continuous for each $g\in L_2$, 
it follows that the map $t\mapsto \|\widetilde U_tf\|_\infty$
is lower semicontinuous, hence it is measurable on $\Ri$.
By the proof of Theorem~\ref{tkoop206}, we deduce that the set 
$\{\widetilde U_t : t\in[2,3]\}$ is
bounded in $\cl(L_\infty)$.
Since $(\widetilde U_t)_{t\in\Ri}$ is a one-parameter group on $L_\infty$, also
$\{\widetilde U_t : t \in [-1,1]\}$ is bounded in $\cl(L_\infty)$.
Let  $c:=\sup\{\|\widetilde U_t\|_{\infty\to\infty} : t \in [-1,1] \} $.
Then $\|\psi_t\|_\infty = \|\widetilde U_t\one\|_\infty \leq c$ for all $t\in[-1,1]$.
Hence $\sup \{ \|U_t\| : t \in [-1,1] \} < \infty$.
Since $L_\infty$ is dense in $L_2$, it follows from (\ref{epkoop401;1}) that 
$U$ is a $C_0$-group.
\end{proof}

\begin{cor} \label{ckoop400.7}
Every differentiable cocycle is a $C_0$-cocycle.
Every coboundary is a $C_0$-cocycle.
\end{cor}

\begin{prop} \label{pkoop401}
Let $\zeta\in L_2(X)$.
Then there exists at most one function $\psi \colon \Ri\to L_\infty(X)$ 
such that $\psi$ is a cocycle over $V$ and the cocycle $\psi$ is differentiable 
with derivative $\zeta$.
\end{prop}
\begin{proof}
Let $\psi,\widetilde \psi \colon \Ri\to L_\infty$ be 
cocycles over $V$ which are differentiable with derivative~$\zeta$.
Then $\psi_0 = \one = \widetilde \psi_0$ a.e.
Let $U$ be the group associated with $\psi$.
Then $U$ is a $C_0$-group by Proposition~\ref{pkoop400.5}.
Moreover, $\sup \{ \|\psi_t\|_\infty : t \in [-1,1] \} < \infty$
by (\ref{etkoop206;10}) in Theorem~\ref{tkoop206}, or the proof of 
Proposition~\ref{pkoop400.5}.
In addition, $\psi_t \neq 0$ a.e.\ and $\frac{1}{\psi_t} = \psi_{-t} \circ T_t$
for all $t \in \Ri$ by (\ref{eSkoop3;3}).
Define $\eta \colon \Ri \to L_\infty$ by
$\eta(t) = \eta_t:= \frac{\widetilde{\psi}_t}{\psi_t}$.
Then $\eta_{t+t'}=\eta_t\cdot(\eta_{t'}\circ T_t)$ for all $t,t'\in\Ri$,
so $\eta$ is a cocycle over $V$.
Moreover,
\begin{eqnarray*}
\frac{1}{t}\left(\eta_t-\one\right)
& = & \frac{1}{t}
\frac{\widetilde{\psi}_t-\psi_t}{\psi_t}  
= \frac{1}{\psi_t}\left(\frac{\widetilde{\psi}_t-\one}t-\frac{\psi_t -\one}{t}\right)
=(\psi_{-t}\circ T_t)\left(\frac{\widetilde{\psi}_t-\one}{t}-\frac{\psi_t-\one}{t}\right)
\end{eqnarray*}
for all $t \in \Ri \setminus \{ 0 \} $.
Since $\sup \{ \|\psi_{-t}\|_\infty : t \in [-1,1] \} < \infty$, one deduces that
the cocycle $\eta$ is differentiable and $\lim_{t\to0}\frac{1}{t}(\eta_t-\one)=0$ in $L_2$.

Let $t\in\Ri$ and $h\in\Ri\setminus\{0\}$.
Then
\begin{eqnarray*}
\frac{1}{h}(\eta_{t+h}-\eta_t)
& = & \frac{1}{h}(\eta_t\cdot(\eta_h\circ T_t)-\eta_t)  \\
& = & \eta_t\cdot\left(\frac{\eta_h-\one}h\circ T_t\right)
= \eta_t\cdot V_t\left(\frac{\eta_h-\one}h\right).
\end{eqnarray*}
It follows that $\lim_{h\to 0}\frac{1}{h}(\eta_{t+h}-\eta_t)=0$ in $L_2$.
Therefore $\eta$ is differentiable from $\Ri$ into $L_2$ and $\eta'(t)=0$
for all $t\in\Ri$.
So $\eta$ is constant and $\eta(t)=\eta(0)=\one$
for all $t\in\Ri$.
Hence $\widetilde{\psi}_t=\psi_t$ for all $t\in\Ri$, which completes the proof.
\end{proof}

\begin{lemma}\label{lkoop404}
Let $\zeta\in L_\infty(X)$.
Define $\psi \colon \Ri \to L_\infty(X)$ by $\psi_t :=e^{\int_0^t\zeta\circ T_s\,ds}$.
Then $\psi$ is a differentiable cocycle with derivative $\zeta$.
\end{lemma}
\begin{proof}
We first show that $\psi$ is a cocycle over $V$.
Let $t,t' \in \Ri$.
Then 
\begin{eqnarray*}
\int_0^{t+t'} \zeta \circ T_s \, ds
& = & \int_0^t \zeta \circ T_s \, ds + \int_t^{t+t'} \zeta \circ T_s \, ds  \\
& = & \int_0^t \zeta \circ T_s \, ds + \int_0^{t'} \zeta \circ T_{s+t} \, ds  \\
& = & \int_0^t \zeta \circ T_s \, ds + \Big( \int_0^{t'} \zeta \circ T_s \, ds \Big) \circ T_t .
\end{eqnarray*}
Hence $\psi$ is a cocycle over $V$.

Next we show that $\psi$ is differentiable.
Recall that $|e^z-1-z| \leq |z|^2e^{|z|}$
for all $z\in\Ci$.
Let $t \in [-1,1] \setminus \{ 0 \} $.
Then 
\begin{eqnarray*}
\left|\frac{\psi_t-\one}{t}-\zeta\right|
& \leq &
\left|\frac{\psi_t-\one - \int_0^t\zeta\circ T_s\,ds}{t} \right|
 + \left|\frac{1}{t} \int_0^t\zeta\circ T_s\,ds - \zeta\right|  \\
& \leq & \frac{1}{|t|} \left(\int_0^t\|\zeta\|_\infty\right)^2 \, e^{|t| \, \|\zeta\|_\infty}
   + \Big| \frac{1}{t} \int_0^t |\zeta\circ T_s-\zeta| \, ds \Big| \\
& \leq & |t| \, \|\zeta\|_\infty^2 \, e^{\|\zeta\|_\infty}
   + \Big| \frac{1}{t} \int_0^t |V_s \zeta - \zeta| \, ds \Big|
.
\end{eqnarray*}
Therefore 
\[
\Big\| \frac{1}{t} (\psi_t-\one) - \zeta \Big\|_2
\leq |t| \, \|\zeta\|_\infty^2 \, e^{\|\zeta\|_\infty}
   + \Big| \frac{1}{t} \int_0^t \|V_s \zeta - \zeta\|_2 \, ds \Big|
.  \]
Since $s \mapsto \|V_s \zeta - \zeta\|_2$ is continuous,
one deduces that the cocycle $\psi$ is differentiable with derivative~$\zeta$.
\end{proof}

\begin{lemma}\label{lkoop405}
Let $\zeta\in L_\infty(X)$ and let $\psi \colon \Ri\to L_\infty(X)$ be a cocycle.
Then $\psi$ is differentiable with derivative $\zeta$
if and only if 
$\psi_t=e^{\int_0^t\zeta\circ T_s\,ds}$ for all $t\in\Ri$.
\end{lemma}
\begin{proof}
This follows immediately from Proposition~\ref{pkoop401} and Lemma~\ref{lkoop404}.
\end{proof}

Next we turn to coboundaries.

\begin{lemma} \label{lkoop403}
Let $\psi$ be a coboundary with transfer function $\theta$.
\begin{tabel}
\item \label{lkoop403-1}
The coboundary $\psi$ is differentiable if and only if $\theta \in D(B)$.
Moreover, if $\psi$ is differentiable, then the derivative is $\frac{B \theta}{\theta}$.
\item \label{lkoop403-2}
If $\theta \in D(B)$ and $B \theta \in L_\infty(X)$, then
\[
\psi_t = e^{\int_0^t \frac{B \theta}{\theta} \circ T_s \, ds}
\]
for all $t \in \Ri$.
\end{tabel}
\end{lemma}
\begin{proof} 
If $t \in \Ri \setminus \{ 0 \} $, then
\[
\frac{\psi_t-\one}{t} 
= \frac{1}{\theta} \cdot\frac{\theta\circ T_t-\theta}{t}
= \frac{1}{\theta} \cdot\frac{V_t \theta - \theta}{t}
. \]
Hence $\psi$ is differentiable if and only if $\theta \in D(B)$.
Moreover, if $\psi$ is differentiable, then the derivative is $\frac{1}{\theta} \, B \theta$.
This proves Statement~\ref{lkoop403-1}.

If $\theta \in D(B)$ and $B \theta \in L_\infty$, then 
$\zeta = \frac{1}{\theta} \, B \theta \in L_\infty$.
Now Statement~\ref{lkoop403-2} follows from Lemma~\ref{lkoop405}.
\end{proof}

Note that Example~\ref{xkoop312}  yields a $C_0$-cocycle 
(in fact a coboundary) which is not differentiable.
It also gives an example of a coboundary which is not a 
coboundary with an $L_\infty$-differentiable transfer function.

\begin{lemma} \label{lkoop407} 
Let $\psi$ be a differentiable cocycle with derivative $\zeta\in L_\infty(X)$.
Then the following conditions are equivalent.
\begin{tabeleq}
\item \label{lkoop407-1} 
$\psi$ is a coboundary.
\item \label{lkoop407-2} 
$\psi$ is a coboundary with an $L_\infty$-differentiable transfer function.
\item \label{lkoop407-3} 
There exists a $\theta\in D(B)\cap L_\infty(X)$ such that $\theta\neq0$-a.e., 
$\frac{1}{\theta} \in L_\infty(X)$, $B \theta\in L_\infty(X)$ and
$\zeta=\frac{B \theta}{\theta}$.
\end{tabeleq}
\end{lemma}
\begin{proof} 
The implication \ref{lkoop407-2}$\Rightarrow$\ref{lkoop407-1} is trivial.

`\ref{lkoop407-1}$\Rightarrow$\ref{lkoop407-2}' and `\ref{lkoop407-1}$\Rightarrow$\ref{lkoop407-3}'.
Let $\theta$ be a transfer function of $\psi$.
By definition $\theta \in L_\infty$, $\theta\neq0$-a.e.\ 
and $\frac{1}{\theta} \in L_\infty$.
Then Lemma~\ref{lkoop403}\ref{lkoop403-1} gives $\theta \in D(B)$ and 
$\zeta = \frac{B \theta}{\theta}$.
So $B \theta = \zeta \, \theta \in L_\infty$.

`\ref{lkoop407-3}$\Rightarrow$\ref{lkoop407-1}'.
Define $\widetilde \psi \colon \Ri \to L_\infty$ by 
$\widetilde \psi_t = \frac{\theta \circ T_t}{\theta}$.
Then Lemma~\ref{lkoop403}\ref{lkoop403-1} implies that the coboundary
$\widetilde \psi$ is differentiable with derivative $\zeta$.
By the uniqueness of Proposition~\ref{pkoop401} one deduces that $\psi = \widetilde \psi$.
So $\psi$ is a coboundary.
\end{proof}

We now give an example of a differentiable cocycle which is not a coboundary.

\begin{exam} \label{xkoop408}
Let $\Ti = \{ z \in \Ci : |z| = 1 \} $ be the torus with normalized Haar measure.
For all $t \in \Ri$ define $T_t \colon \Ti \to \Ti$ by $T_t z = e^{it} \, z$ and 
define $V_t \colon L_2(\Ti) \to L_2(\Ti)$ by $V_t f = f \circ T_t$.
Then $V = (V_t)_{t \in \Ri}$ is a $C_0$-group.
Fix $\zeta \in L_\infty(\Ti)$ with $\int \zeta \not\in i \Zi$.
Define $\psi \colon \Ri \to L_\infty(\Ti)$ by 
\[
\psi_t = e^{\int_0^t \zeta \circ T_s \, ds}
.  \]
Then $\psi$ is a differentiable cocycle by Lemma~\ref{lkoop404}.
Now suppose that $\psi$ is a coboundary.
Let $\theta$ be a transfer function.
Then $\psi_{2\pi} = \frac{\theta \circ T_{2\pi}}{\theta} = \one$.
Hence $\int_0^{2\pi} \zeta \circ T_s \, ds \in 2 \pi i \Zi$ a.e.
But $\int_0^{2\pi} \zeta \circ T_s \, ds = 2 \pi \int_\Ti \zeta$ a.e.
So $\int_\Ti \zeta \in i \Zi$.
This is a contradiction.
\end{exam}

\subsection*{Acknowledgements}
We wish to thank Yuri Tomilov for making a reference to \cite{ABHN} Lemma~3.16.4,
which then led to the proof of Theorem~\ref{tkoop206}.
In addition we wish to thank him for the reference to \cite{GalP}
and many critical comments
which improved the paper.
The authors are most grateful for the hospitality and 
fruitful stay of the first-named author at the University of Toru\'n
and the second-named author at the University of Auckland.
Part of this work is supported by an
NZ-EU IRSES counterpart fund and the Marsden Fund Council from Government funding,
administered by the Royal Society of New Zealand.
Part of this work is supported by the the NCN grant DEC-2011/03/B/ST1/00407 and 
the EU Marie Curie IRSES program, project `AOS', No.~318910.

\end{document}